\documentclass[11pt,final]{siamltex}
\usepackage{amsmath}
\usepackage{bm}
\usepackage{amssymb,version}
\usepackage{cases}
\usepackage{color}
\usepackage{verbatim}
\usepackage{graphicx}

\allowdisplaybreaks
\begin{document}
   \title{ On efficient linear and fully decoupled finite difference method for wormhole propagation with heat transmission process on staggered grids  
\thanks{This work is supported by the National Natural Science Foundation of China  grants  12271302, 12131014. 
}}

    \author{ Xiaoli Li
        \thanks{School of Mathematics, Shandong University, Jinan, Shandong, 250100, P.R. China. Email: xiaolimath@sdu.edu.cn}.
        \and Ziyan Li
         \thanks{Department of Mathematics, City University of Hong Kong, Hong Kong SAR, China. Email: ziyali3-c@my.cityu.edu.hk}.
         \and Hongxing Rui
        \thanks{Corresponding author. School of Mathematics, Shandong University, Jinan, Shandong, 250100, P.R. China. Email: hxrui@sdu.edu.cn}.
}

\maketitle

\begin{abstract}
In this paper, we construct an efficient linear and fully decoupled finite difference scheme for wormhole propagation with heat transmission process on staggered grids, which only requires solving a sequence of linear elliptic equations at each time step. We first derive the positivity preserving properties for the discrete porosity and its difference quotient in time, and then obtain optimal error estimates for the velocity, pressure, concentration, porosity and temperature in different norms rigorously and carefully by establishing several auxiliary lemmas for the highly coupled nonlinear system. Numerical experiments in two- and three-dimensional cases are provided to verify our theoretical results and illustrate the capabilities of the constructed method.
\end{abstract}

\markboth{XIAOLI LI, ZIYAN LI AND HONGXING RUI} {Wormhole Propagation with Heat Transmission Process}
 \section{Introduction}
 As an efficient technique of enhanced oil recovery, the acid treatment of carbonate reservoirs has been widely used in improving oil production rate. In this technique, acid is injected into matrix to dissolve the rocks and deposits around the well bore, which facilitates oil flow into production well, and thus forms a channel with high porosity. However, the efficiency of this technique strongly depends the dissolution patterns.
Specifically three dissolving patterns can be observed with the increase of injection rate, which include face dissolution pattern, wormhole pattern and uniform dissolution pattern.  Wormhole pattern with narrow channel is the most efficient one for successful simulation \cite{fredd1998influence,golfier2002ability,liu2013wormhole,chen2018simulation}. 


Due to the important role that wormhole plays in enhancing  productivity, several research works have been conducted to study the formation and propagation of wormholes \cite{panga2005two,wu2015parallel,golfier2001acidizing,akanni2017computational,wei2017modeling}. 
McDuff et al. \cite{mcduff2010understanding} developed a new methodology to give high-resolution nondestructive imaging and analysis in the experimental studies for the wormhole model. Pange et al. \cite{panga2005two} proposed the well-known two-scale continuum model for this problem in reactive dissolution of a porous medium. There are also many numerical research works for the wormhole model. Kou et al. \cite{kou2016mixed} developed a mixed finite element method and established stability analysis and a priori error estimates for velocity, pressure, concentration and porosity in different norms. They \cite{kou2019semi} also proposed a semi-analytic time scheme for the wormhole propagation with the Darcy-Brinkman-Forchheimer model. Li et al. \cite{li2017characteristic,li2018block,li2019superconvergence} extended finite difference methods on the staggered grids to the wormhole models with different frameworks. Later the discontinuous Galerkin method was applied to the wormhole model in \cite{guo2019high}.
 Recently Xu \cite{xu2019high} constructed the high-order bound-preserving discontinuous Galerkin method for this problem to preserve the boundedness of porosity and concentration of acid. However, the above works do not consider the factor of temperature, which has an important influence on the thermodynamic parameters including the surface reaction rate and molecular diffusion coefficient \cite{kalia2010fluid,wu2021thermodynamically}. As far as we know, most previous numerical works have focused only on  the chemical reaction and mass transport processes in wormhole propagation, but completely ignored the significant influence of temperature factor. There are only a few works to consider the wormhole model with heat transmission process. Kalia et al. \cite{kalia2010fluid} applied a mathematical model to investigate the effect of temperature on carbonate matrix acidizing. They also presented the numerical simulation for the wormhole model by using the finite volume method. A radial heat transfer model is introduced to capture heat transfer and reaction heat in \cite{li2017simulation}. Recently Wu et al. \cite{wu2021thermodynamically} proposed the modified momentum conservation equation and established the thermal DBF framework  by introducing the energy balance equation. However, to the best of our knowledge, there are no related work to consider the theoretical analysis for wormhole propagation with heat transmission process.  It is much more challenging to develop efficient numerical schemes and to carry out corresponding error analysis for this highly coupled nonlinear system. 
 
The main purposes of this work are to construct an efficient linear and fully decoupled finite difference scheme for wormhole propagation with heat transmission process on staggered grids, and carry out error analysis rigorously. We also give several 
numerical experiments in two- and three-dimensional cases to verify our theoretical results and illustrate the capabilities of the constructed method. More precisely, the work presented in this paper is unique in the following aspects: 

(i) Efficient linear and fully decoupled scheme for this highly coupled nonlinear system is proposed by introducing auxiliary variables $\textbf{w}$ and $\textbf{v}$, and using the implicit-explicit discretization. The constructed scheme only requires solving a sequence of linear elliptic equations at each time step; 

(ii) We first derive the positivity preserving properties for the discrete porosity and its difference quotient in time, and then handle with the complication resulted from the fully coupling relation of multivariables, including porosity, pressure, velocity, solute concentration and temperature by establishing several auxiliary lemmas.

(iii) The optimal error analysis for the velocity, pressure, concentration, porosity and temperature in different norms is established.  We believe that our error analysis for the constructed fully decoupled and linear scheme is the first work. 

The paper is organized as follows. In Section 2 we describe mathematical model. In Section 3 we construct finite difference method on staggered grids. In Section 4 we carry out error estimates for the discrete scheme. In Section 5, we present numerical experiments in two- and three-dimensional cases to verify our theoretical results and illustrate the capabilities of the constructed method. In Section 6 we give some concluding remarks.
 
 \section{Mathematical model}
  In this paper, we consider a heat-transfer model to describe the temperature behavior for  wormhole propagation by using the two-scale continuum model \cite{li2017simulation,kalia2010fluid,panga2005two}, which is established by coupling local pore-scale phenomena to macroscopic variables (Darcy velocity, pressure and concentration) through structure-property relationships (permeability-porosity, interfacial area-porosity, and so on).
 \subsection{Darcy scale model}
 The Darcy scale model equations are given by 
\begin{numcases}{}
\gamma\frac{\partial p}{\partial t}+\frac{\partial \phi}{\partial t}+\nabla\cdot \textbf{u}=
f,\ \ \textbf{x}\in\Omega,\ t\in J,\label{e_Darcy1} \\
\textbf{u}=-\dfrac{K(\phi)}{\mu}\nabla p,~~~\textbf{x}\in\Omega,~t\in J,\label{e_Darcy2} \\
\dfrac{\partial (\phi c_f)}{\partial t}+\nabla\cdot (\textbf{u}c_f)
-\nabla\cdot (\phi\textbf{D} \nabla c_f)=k_ca_v(c_s-c_f)+f_Pc_f+f_Ic_I,
\ \ \textbf{x}\in\Omega, \ t\in J,\label{e_Darcy3} \\
\frac{\partial \phi}{\partial t}=\frac{ R(c_f,T) a_v \alpha }{\rho_s}, \ \ \textbf{x}\in\Omega, \ t\in J,\label{e_Darcy4} \\
R(c_f,T) = k_c (c_f-c_s), \ \ \textbf{x}\in\Omega, \ t\in J,\label{e_Darcy5}
\end{numcases} 
where $\Omega$ is an open bounded domain. $J=(0,Q]$, and $Q$ denotes the final time. $p$ is the pressure, $\mu$ is the fluid viscosity, $\textbf{u}$ is the Darcy velocity of the fluid, $f=f_I+f_P$,
$f_P$ and $f_I$ are production and injection rates respectively. $\gamma$ is a pseudo-compressibility parameter
that results in slight change of the density of the fluid phase in the dissolution process. $\phi$ and $K$ are the porosity and permeability of the rock respectively,
$c_f$ is the cup-mixing concentration of the acid in the fluid phase. $c_I$ is the injected concentration.
For simplicity we assume that diffusion coefficient $\textbf{D}(\textbf{x})=d_{mol}\textbf{I}
=diag(D_{ll}),~l=1,2$ is diagonal matrix in the following.
$k_c$ is the local mass-transfer coefficient, $a_v$ is the interfacial area available for reaction per unit volume of the medium.
The variable $c_s$ is the concentration of the acid at the fluid-solid interface, and the relationship between $c_f$ and $c_s$ can be described as follows.
\begin{equation}\label{e_Darcy6}
c_s=\frac{c_f}{1+k_s(T)/k_c},
\end{equation}
where the surface reaction rate $k_s$ is a function of the temperature \cite{wu2021thermodynamically,li2017simulation}, no longer deemed as a constant in \cite{li2019superconvergence}.   Here we assume that $ k_{s1}< k_s(T) < k_{s2}$ and $k_s(T)$ is Lipschitz continuous for simplicity.
$\alpha$ is the dissolving power of the acid and $\rho_s$ is the density of the solid phase.

 \subsection{Pore scale model}
The pore-scale model uses structure property relations to describe the changes in permeability and interfacial surface area as dissolution occurs. The relationship between the permeability and the porosity is established by the Carman-Kozeny correlation:
 \begin{equation}\label{e_Pore1}
\frac{K}{K_0}=\frac{\phi}{\phi_0}\left(\frac{\phi(1-\phi_0)}{\phi_0(1-\phi)}\right)^2,
\end{equation}
where $\phi_0$ and $K_0$ are the initial porosity and permeability of the rock respectively.  Using porosity and permeability, $a_v$ is shown as
\begin{equation}\label{e_Pore2}
\frac{a_v}{a_0}=\frac{\phi}{\phi_0}\sqrt{\frac{K_0\phi}{K\phi_0}},
\end{equation}
where $a_0$ is the initial interfacial area.
 
 \subsection{Heat-transfer model}
 In this paper, a heat-transfer model is introduced to determine
the temperature behavior during wormhole propagation, which considers heat conduction, heat convection and reaction heat \cite{li2017simulation, wu2021thermodynamically}. 
\begin{equation}\label{e_Heat1}
\frac{ \partial [ \left( \rho_s(1-\phi) \theta_{s} + \rho_f  \phi \theta_{f} \right) T ]}{ \partial t } +  \nabla \cdot (\rho_f \theta_{f} \textbf{u}  T) = \nabla \cdot ( \lambda(\phi) \nabla T ) + a_v(\phi) H_r(T)R(c_f, T) ,
\end{equation}
where $T$ is the temperature, $\rho_s$ and $\rho_f$ are the density of rock and acid respectively. $\theta_{s}$ and $\theta_{f}$ are the heat capacities of rock and acid respectively. The average thermal conductivity between acid solution and rock $\lambda(\phi)= (1-\phi) \lambda_s + \phi \lambda_f $, where $ \lambda_s $ and $\lambda_f $ are the thermal conductivities of rock and acid respectively.   Here we assume that the reaction heat $H_r(T)$ is Lipschitz continuous for simplicity.
\eqref{e_Heat1} establishes the heat transmission process during acid
injection. The first term in this equation describes the variation of temperature over time, the second term represents thermal convection due to the acid flow during wormhole propagation. The first term on the right hand side describes thermal conduction, and the last term represents the reaction heat generation rate. 

In this paper, we assume that the temperature of the acid and matrix can be represented by a single notation $T$, since the speed of heat transfer is much faster than the fluid speed. i. e. the temperature of the acid and matrix is the same when acid is injected into the matrix. Differentiating between the acid temperature and matrix temperature would cause great difficulty to theories and applications, and the details of the heat transfer between the acid and matrix must be researched carefully. In fact, due to the geothermal factor, the initial matrix temperature may be higher than the injected acid temperature. Relevant work can be left to the future.

 \subsection{Boundary and initial conditions}
 In this paper, the boundary and initial conditions are as follows.
\begin{equation}\label{e5}
  \left\{
   \begin{array}{l}
 \textbf{u}\cdot \textbf{n}=0,\ \phi\textbf{D}\nabla c_f \cdot \textbf{n}=0, \ \lambda \nabla T \cdot \textbf{n}=0, \ \ \textbf{x}\in \partial\Omega,\ t \in J,\\
p(\textbf{x},0)=p_0(\textbf{x}),\ \ \textbf{x}\in\Omega,\\
c_f(\textbf{x},0)=c_{f0}(\textbf{x}), \ \ \textbf{x}\in\Omega,\\
\phi(\textbf{x},0)=\phi_0(\textbf{x}), \ \ \textbf{x}\in\Omega, \\
T (\textbf{x},0)= T_0(\textbf{x}), \ \ \textbf{x}\in\Omega,
   \end{array}
   \right.
  \end{equation}
  where $\textbf{n}$ is the unit outward normal vector of the domain $\Omega$.

 \section{Finite difference method on staggered grids} 
 In this section, we consider the finite difference method for the coupled system on staggered grids.  
To fix the idea, we consider $\Omega=(L_{lx},L_{rx})\times (L_{ly},L_{ry})$. Three dimensional rectangular domains can be dealt with similarly. The grid points are denoted by
$$(x_{i+1/2},y_{j+1/2}),~~~i=0,...,N_x,~~j=0,...,N_y,$$
and the notations similar to those in \cite{Weiser1988} are used.
\begin{equation*}
\aligned
&x_{i}=(x_{i-\frac{1}{2}}+x_{i+\frac{1}{2}})/2,~~~i=1,...,N_x,\\
&h_{i}^x=x_{i+\frac{1}{2}}-x_{i-\frac{1}{2}},~~~i=1,...,N_x, \\
&h_{i+\frac{1}{2}}^x=x_{i+1}-x_{i}=(h_{i}^x+h_{i+1}^x)/2,~~~i=1,...,N_x-1,\\
&y_{j}=(y_{j-\frac{1}{2}}+y_{j+\frac{1}{2}})/2,~~~j=1,...,N_y,\\
&h_{j}^y=y_{j+\frac{1}{2}}-y_{j-\frac{1}{2}},~~~j=1,...,N_y,\\
&h_{j+\frac{1}{2}}^y=y_{j+1}-y_{j}=(h_{j}^y+h_{j+1}^y)/2,~~~j=1,...,N_y,\\
&h=\max\limits_{i,j}\{h_{i}^x,h_j^y\}.
\endaligned
\end{equation*}
Let $g_{i,j},~g_{i+\frac{1}{2},j},~g_{i,j+\frac{1}{2}}$ denote $g(x_{i},y_{j}),~g(x_{i+\frac{1}{2}},y_{j}),~g(x_{i},y_{j+\frac{1}{2}}). $ Define the discrete inner products and norms as follows,
\begin{equation*}
\aligned
&(f,g)_{M}=\sum\limits_{i=1}^{N_{x}}\sum\limits_{j=1}^{N_{y}}h_{i}^xh_{j}^yf_{i,j}g_{i,j},\\
&(f,g)_{x}=\sum\limits_{i=1}^{N_{x}-1}\sum\limits_{j=1}^{N_{y}}h_{i+\frac{1}{2}}^xh_{j}^yf_{i+\frac{1}{2},j}g_{i+\frac{1}{2},j},\\
&(f,g)_{y}=\sum\limits_{i=1}^{N_{x}}\sum\limits_{j=1}^{N_{y}-1}h_{i}^xh_{j+\frac{1}{2}}^yf_{i,j+\frac{1}{2}}g_{i,j+\frac{1}{2}},\\
&(\textit{\textbf{v}},\textit{\textbf{r}})_{TM}=(v^x,r^x)_{x}+(v^y,r^y)_{y}.
\endaligned
\end{equation*}
For simplicity from now on we always
omit the superscript $n$ if the omission does not cause conflicts.
Define
\begin{equation*}
\aligned
&[d_{x}g]_{i+\frac{1}{2},j}=(g_{i+1,j}-g_{i,j})/h_{i+\frac{1}{2}}^x,\\
&[d_{y}g]_{i,j+\frac{1}{2}}=(g_{i,j+1}-g_{i,j})/h_{j+\frac{1}{2}}^y,\\
&[D_{x}g]_{i,j}=(g_{i+\frac{1}{2},j}-g_{i-\frac{1}{2},j})/h_{i}^x,\\
&[D_{y}g]_{i,j}=(g_{i,j+\frac{1}{2}}-g_{i,j-\frac{1}{2}})/h_{j}^y,\\
&[d_{t}g]^n_{i,j}=(g_{i,j}^n-g_{i,j}^{n-1})/\Delta t.
\endaligned
\end{equation*}
 For simplicity we only consider the case that $h_{i+1/2}=h,\ k_{j+1/2}=k$, i.e. uniform meshes are used both in $x$ and $y$-directions. 
 
 Define $\textbf{w}=(w^x,w^y)=\textbf{u}c_f - \phi\textbf{D}\nabla c_f $ and $\textbf{v}=(v^x,v^y)= \rho_f \theta_{f} \textbf{u}  T - \lambda(\phi) \nabla T $, then  (\ref{e_Darcy3}) and (\ref{e_Heat1}) can be transformed into
  \begin{equation} \label{e_Darcy3_transform}
\aligned
& \dfrac{\partial (\phi c_f)}{\partial t}+\nabla\cdot \textbf{w} =k_c a_v(\phi)
\left( \frac{1}{ 1+ k_s(T) /k_c } -1 \right) c_f
+f_Pc_f+f_Ic_I,
\endaligned
\end{equation}
and  
 \begin{equation}\label{e_Heat1_transform}
\frac{ \partial [ \left( \rho_s(1-\phi) \theta_{s} + \rho_f  \phi \theta_{f} \right) T ]}{ \partial t } +  \nabla \cdot \textbf{v} = a_v(\phi) H_r(T)R(c_f, T) .
\end{equation}

 Set  $\Delta t=Q/N,\ t^n=n\Delta t,$  for $ n\leq N $, 
and define
$ \displaystyle [d_{t}f]^n=\frac{f^n-f^{n-1}}{\Delta t}$. Denote by $\{\Psi^n, P^n, \textbf{U}^n,  C_f^n,  \textbf{W}^n, Z^n,  \textbf{V}^n \}_{n=1}^{N}$, the approximations to $\{\phi^n, p^n, \textbf{u}^n, c_f^n, \textbf{w}^n , T^n, \textbf{v}^n\}_{n=1}^{N}$ respectively, with the boundary approximations 
\begin{equation}\label{e_boundary}
  \left\{
   \begin{array}{l}
   \displaystyle U_{1/2,j}^{x,n}=U_{N_x+1/2,j}^{x,n}=0,\quad 1\leq j\leq N_y,\\
   \displaystyle U_{i,1/2}^{y,n}=U_{i,N_y+1/2}^{y,n}=0,\quad 1\leq i\leq N_x,\\
   \displaystyle W_{1/2,j}^{x,n}=W_{N_x+1/2,j}^{x,n}=0,\quad 1\leq j\leq N_y,\\
   \displaystyle W_{i,1/2}^{y,n}=W_{i,N_y+1/2}^{y,n}=0,\quad 1\leq i\leq N_x,\\
   \displaystyle V_{1/2,j}^{x,n}=V_{N_x+1/2,j}^{x,n}=0,\quad 1\leq j\leq N_y,\\
   \displaystyle V_{i,1/2}^{y,n}=V_{i,N_y+1/2}^{y,n}=0,\quad 1\leq i\leq N_x,
   \end{array}
   \right.
  \end{equation}
and the initial approximations for $1\leq i\leq N_x,1\leq j\leq N_y$, 
\begin{equation}\label{e_initial}
  \left\{
   \begin{array}{l}
   \displaystyle P_{i,j}^0=p_{0,i,j}, \  C_{f,i,j}^0=c_{f0,i,j}, \\
   \displaystyle Z_{i,j}^0=T_{0,i,j}, \  \Psi_{i,j}^0=\phi_{0,i,j} .
   \end{array}
   \right.
  \end{equation}
Then, the  fully discrete  scheme based on the finite difference method on staggered grids is as follows: 
\begin{subequations}\label{e_model_fully_discrete}
    \begin{align}
    & \gamma [d_t P]^{n+1}_{i,j} + [d_t \Psi ]^{n+1}_{i,j} + [D_xU]^{x,n+1}_{i,j} +  [D_yU]^{y,n+1}_{i,j} = f^{n+1}_{i,j},  \label{e_model_fully_discrete_Darcy} \\   
    & U^{x,n+1}_{i+1/2,j} = - \frac{K(\Pi_h \Psi^{n+1}_{i+1/2,j} ) }{ \mu } [d_x P]^{n+1}_{i+1/2,j} , \ \  U^{y,n+1}_{i,j+1/2} = - \frac{K(\Pi_h \Psi^{n+1}_{i,j+1/2} ) }{ \mu } [d_y P]^{n+1}_{i,j+1/2} ;
     \label{e_model_fully_discrete_Velocity} \\  
    & [d_t (\Psi C_f) ]^{n+1}_{i,j} +  [D_x W]^{x,n+1}_{i,j} +  [D_y W]^{y,n+1}_{i,j} \notag \\
    & \ \ \ \ \ \ \ \ \ \ \ \ \ \ \ 
    = k_c a_v( \Psi^{n+1}_{i,j} ) \left( \frac{1}{ 1+ k_s(Z^n_{i,j} ) /k_c } -1 \right) C^{n+1}_{f,i,j} + f_P^{n+1} C^{n+1}_{f,i,j} + f_I^{n+1} c_I^{n+1},      \label{e_model_fully_discrete_concentration} \\  
        & W^{x,n+1}_{i+1/2,j} = U^{x,n+1}_{i+1/2,j} \Pi_h C^{n+1}_{f,i+1/2,j} - \Pi_h \Psi^{n+1}_{i+1/2,j} D_{11} [d_x C_f]^{n+1}_{i+1/2,j}, \notag \\
        & W^{y,n+1}_{i,j+1/2}  = U^{y,n+1}_{i,j+1/2}  \Pi_h C^{n+1}_{f,i,j+1/2}  - \Pi_h \Psi^{n+1}_{i,j+1/2}  D_{22} [d_y C_f ]^{n+1}_{i,j+1/2} ;      \label{e_model_fully_discrete_auxiliary W} \\  
     &  [ d_t \big( \left(  \rho_s(1- \Psi) \theta_{s} + \rho_f \Psi \theta_{f} \right) Z \big) ]^{n+1}_{i,j} +  [D_xV]^{x,n+1}_{i,j} +  [D_yV]^{y,n+1}_{i,j} = a_v( \Psi^{n+1}_{i,j} ) H_r(Z^n_{i,j}) R(  \overline{C}^{n+1}_{f,i,j},  Z^n_{i,j} ),   \label{e_model_fully_discrete_Temperature} \\  
         & V^{x,n+1}_{i+1/2,j} = \rho_f \theta_{f} U^{x,n+1}_{i+1/2,j} \Pi_h Z^{n+1}_{i+1/2,j} - \lambda ( \Pi_h \Psi^{n+1}_{i+1/2,j}) [d_x Z]^{n+1}_{i+1/2,j}, \notag \\ 
        & V^{y,n+1}_{i,j+1/2}  = \rho_f \theta_{f} U^{y,n+1}_{i,j+1/2}  \Pi_h Z^{n+1}_{i,j+1/2}  - \lambda ( \Pi_h \Psi^{n+1}_{i,j+1/2}  ) [d_y Z]^{n+1}_{i,j+1/2}  ; \label{e_model_fully_discrete_auxiliary V} 
\end{align}
\end{subequations}   
where $\Pi_h$ is an interpolation operator with second-order or higher precision. 

Using \eqref{e_Darcy4}-\eqref{e_Pore2}, we have 
\begin{equation}\label{e_model_porosity}
\frac{ \partial \phi }{ \partial t } =  \frac{ \alpha k_c a_0 }{ \rho_s } \left( 1-  \frac{ 1 }{ 1+ k_s(T)/k_c } \right) \frac{ 1-\phi }{ 1- \phi_0 } c_f.
\end{equation}
For the calculation of the discrete porosity $\Psi$, we use the following scheme.
\begin{equation}\label{e_model_fully_discrete_porosity}
  [d_t \Psi ]^{n+1}_{i,j} = \frac{ \alpha k_c a_0 }{ \rho_s } \left( 1- \frac{ 1}{1+ k_s(Z^n_{i,j} )/k_c } \right) \frac{ 1- \Psi^{n+1}_{i,j} }{ 1-\Psi^0_{i,j} } \overline{ C }^{n}_{f,i,j}, 
\end{equation}
where $\overline{ C }^{n}_{f,i,j} =\max\left\{0,\min\{ C^{n}_{f,i,j} ,1\}\right\}.$

 The difference method will consist of four parts: 
 
(\romannumeral1) If the approximate concentration $C_{f,i,j}^{n}$ and porosity $\Psi_{i,j}^{n},\ n=0,\cdots,N-1$
are known, equation (\ref{e_model_fully_discrete_porosity}) will be used to obtain
a new porosity $\Psi_{i,j}^{n+1}$. 

(\romannumeral2) By using difference scheme (\ref{e_model_fully_discrete_Darcy}) and (\ref{e_model_fully_discrete_Velocity}), an approximation $P_{i,j}^{n+1}$
to the pressure will be calculated using $\Psi_{i,j}^{n+1}$, and then the approximate velocity $U^{x,n+1}_{i+1/2,j}$ and $U^{y,n+1}_{i,j+1/2}$ will be evaluated. 

(\romannumeral3) A new concentration $C_{f,i,j}^{n+1}$ will be calculated using $U^{x,n+1}_{i+1/2,j}$, $U^{y,n+1}_{i,j+1/2}$,  $Z_{i,j}^{n}$ and $\Psi_{i,j}^{n+1}$ in \eqref{e_model_fully_discrete_concentration}-\eqref{e_model_fully_discrete_auxiliary W}, then we get the approximations $W^{x,n+1}_{i+1/2,j}$ and $W^{y,n+1}_{i,j+1/2}$ by using  \eqref{e_model_fully_discrete_auxiliary W}. 

(\romannumeral4)
 A new temperature $Z_{i,j}^{n+1}$ will be calculated in \eqref{e_model_fully_discrete_Temperature} by using $U^{x,n+1}_{i+1/2,j}$, $U^{y,n+1}_{i,j+1/2}$, $C_{f,i,j}^{n+1}$ and $\Psi_{i,j}^{n+1}$, then we get the approximations $V^{x,n+1}_{i+1/2,j}$ and $V^{y,n+1}_{i,j+1/2}$ in \eqref{e_model_fully_discrete_auxiliary V}. 
 
 It is easy to see
that at each time level, the difference scheme has an explicit solution or is a linear
pentadiagonal system with strictly diagonally dominant coefficient matrix, thus the approximate solutions exist uniquely.
   
\section{Error Analysis for the Discrete Scheme}
In this section, we give the error estimates for the fully discrete scheme \eqref{e_model_fully_discrete}-\eqref{e_model_fully_discrete_porosity}.
Set
\begin{equation}\label{Err1}
  \left\{
   \begin{array}{ll}
   \displaystyle
   E_p=P-p, & E_{ \textbf{u} }=(E_{ \textbf{u} }^{x},E_{ \textbf{u} }^{y})=\textbf{U}-\textbf{u},\\
   E_{c_f}= C_f -c_f,&  E_{\textbf{w}}=(E_{\textbf{w}}^{x},E_{\textbf{w}}^{y})=\textbf{W}-\textbf{w},\\
      E_{T}=Z-T, & E_{\textbf{v}}=(E_{\textbf{v}}^{x},E_{\textbf{v}}^{y})=\textbf{V}-\textbf{v},\\
   E_{\phi}=\Psi-\phi.
   \end{array}
   \right.
  \end{equation}
First we present the following lemma which will be used in what follows.
 \medskip
\begin{lemma}\label{le1}
\cite{Weiser1988} Let $q_{i,j},w_{i+1/2,j}^x~and~ w_{i,j+1/2}^y $ be any values such that $w_{1/2,j}^x=w_{N_x+1/2,j}^x=w_{i,1/2}^y=w_{i,N_y+1/2}^y=0$, then
$$(q,D_xw^x)_M=-(d_xq,w^x)_x,$$
$$(q,D_yw^y)_M=-(d_yq,w^y)_y.$$
\end{lemma}  
Next we will prove a priori bounds for the discrete solution $\Psi$ which will be used in what follows.
 \medskip
\begin{lemma}\label{le_porosity_Err1}
Assuming that $0<\phi_{0*}\leq \phi_0\leq \phi_{0}^*<1$, then the discrete porosity $\Psi_{i,j}^n$ is bounded,
 i.e.,
\begin{equation}\label{porosity_Err2}
\aligned
\phi_{0*}\leq \Psi_{i,j}^n<1,\ \ 0\leq i\leq N_x, \ 0\leq j\leq N_y,\ n\leq N.
\endaligned
\end{equation}
It also holds that
 \begin{equation}\label{porosity_Err3}
\aligned
0\leq [d_t\Psi]_{i,j}^n< \frac{ \alpha k_c a_0}{ \rho_s (1- \phi_{0}) }, \ \ 0\leq i\leq N_x, \ 0\leq j\leq N_y, \ n\leq N.
\endaligned
\end{equation}
\end{lemma}

\begin{proof}
The proof is given by induction. It is trivial that $\phi_{0*}\leq \Psi_{i,j}^0<1$.
Suppose that $$\phi_{0*}\leq \Psi_{i,j}^{n-1}<1, \ n\leq N,$$ next we prove that $\Psi_{i,j}^{n}$ also does.

Set $ \displaystyle \beta^{n} = \frac{ \alpha k_c a_0 }{ \rho_s } \left( 1- \frac{ 1}{1+ k_s(Z^n)/k_c } \right) \frac{1}{1-\Psi^0} \overline{ C }_f^{n} \Delta t $ for simplicity. Then we can easily obtain that $ \displaystyle 0<\beta^{n}< \frac{ \alpha k_c a_0}{ \rho_s (1- \phi_{0}) } \Delta t $. By using the definition of $\beta^n$, we can transform \eqref{e_model_fully_discrete_porosity} into the following.
\begin{equation}\label{porosity_Err4}
\aligned
\Psi_{i,j}^{n+1}=\frac{\beta_{i,j}^{n}}{1+\beta_{i,j}^{n}}+\frac{\Psi_{i,j}^{n}}{1+\beta_{i,j}^{n}},
\endaligned
\end{equation}
where we can easily obtain that $\Psi_{i,j}^{n+1}<1$.

Since \eqref{e_model_fully_discrete_porosity} also can be recast as
\begin{equation}\label{porosity_Err5}
\aligned
\Psi_{i,j}^{n+1}-\Psi_{i,j}^{n}=\beta_{i,j}^{n}(1-\Psi_{i,j}^{n+1} ).
\endaligned
\end{equation}
Thus we have that $\Psi_{i,j}^{n+1}>\Psi_{i,j}^{n},$ which leads to the desired results.
\end{proof}

 \medskip
\begin{lemma}\label{le_error_porosity}
 The approximate errors of discrete porosity satisfy
\begin{equation}\label{porosity_Err14}
\aligned
\|E_{\phi}^{m+1} \|^2_{M} \leq & C \Delta t \sum\limits_{n=0}^m  \|E_{c_f}^{n}\|_{M}^2+ C \Delta t \sum\limits_{n=0}^m  \|E_{T}^{n}\|_{M}^2 \\
&+ C \Delta t\sum\limits_{n=0}^m \|E_{\phi}^{n+1}\|_{M}^2+ C (\Delta t)^2, \ \ m \leq N-1,
\endaligned
\end{equation}

\begin{equation}\label{porosity_Err15}
\aligned
\Delta t \sum\limits_{n=0}^m \|d_tE_{\phi}^{n+1} \|_{M}^2 \leq & C \Delta t \sum\limits_{n=0}^m  \|E_{c_f}^{n}\|_{M}^2+ C \Delta t \sum\limits_{n=0}^m  \|E_{T}^{n}\|_{M}^2 \\
&+ C \Delta t\sum\limits_{n=0}^m \|E_{\phi}^{n+1}\|_{M}^2+ C(\Delta t)^2, \ \ m \leq N-1,
\endaligned
\end{equation}
where the positive constant $C$ is independent of $h, k$ and $\Delta t$.
\end{lemma}

\begin{proof}
Subtracting \eqref{e_model_porosity} from \eqref{e_model_fully_discrete_porosity}, we can obtain
\begin{equation}\label{porosity_Err6}
\aligned
d_t E_{\phi,i,j}^{n+1} = & \chi ( \frac{ 1 }{ 1+ k_s(T^{n+1}_{i,j} )/k_c } - 
\frac{ 1}{1+ k_s(Z^n_{i,j} )/k_c } ) \frac{ 1- \Psi^{n+1}_{i,j} }{ 1-\Psi^0_{i,j} } \overline{ C }^{n}_{f,i,j} \\
& + \chi ( 1-  \frac{ 1 }{ 1+ k_s(T^{n+1}_{i,j} )/k_c } ) ( \frac{ 1- \Psi^{n+1}_{i,j} }{ 1-\Psi^0_{i,j} } - \frac{ 1-\phi^{n+1}_{i,j} }{ 1- \phi_{0,i,j} } )  \overline{ C }^{n}_{f,i,j} \\
& + \chi  ( 1-  \frac{ 1 }{ 1+ k_s(T^{n+1}_{i,j} )/k_c } ) \frac{ 1-\phi^{n+1}_{i,j} }{ 1- \phi_{0,i,j} } ( \overline{ C }^{n}_{f,i,j} - c_{f,i,j}^{n+1} ) \\
& + \frac{ \partial \phi }{ \partial t } |^{n+1}_{i,j}- d_t \phi^{n+1}_{i,j},
\endaligned
\end{equation}
where $ \chi = \frac{ \alpha k_c a_0 }{ \rho_s } $.

Multiplying \eqref{porosity_Err6} by $E_{\phi,i,j}^{n+1} h k $ and making summation on $i,j$ for $1 \leq i \leq N_x, \ 1 \leq j \leq N_y $, we have that 
\begin{equation}\label{porosity_Err7}
\aligned
(d_tE_{\phi}^{n+1}, E_{\phi}^{n+1} )_{M}= & \chi  \left( (\frac{ 1 }{ 1+ k_s(T^{n+1} )/k_c } - 
\frac{ 1}{1+ k_s(Z^n )/k_c } ) \frac{ 1- \Psi^{n+1} }{ 1-\Psi^0 } \overline{ C }^{n}_f,  E_{\phi}^{n+1} \right)_{M} \\
& + \chi \left( ( 1-  \frac{ 1 }{ 1+ k_s(T^{n+1} )/k_c } )  ( \frac{ 1- \Psi^{n+1} }{ 1-\Psi^0 } - \frac{ 1-\phi^{n+1} }{ 1- \phi_{0} } )  \overline{ C }^{n}_f,  E_{\phi}^{n+1} \right)_{M} \\
& + \chi  \left(  ( 1-  \frac{ 1 }{ 1+ k_s(T^{n+1} )/k_c } ) \frac{ 1-\phi^{n+1} }{ 1- \phi_{0} } ( \overline{ C }^{n}_f - c_{f}^{n+1} ), E_{\phi}^{n+1} \right)_{M}  \\
& + ( \frac{ \partial \phi^{n+1} }{ \partial t } - d_t \phi^{n+1}, E_{\phi}^{n+1} )_{M} .
\endaligned
\end{equation}
The term on the left side of \eqref{porosity_Err7} can be transformed into
\begin{equation}\label{porosity_Err8}
\aligned
&(d_tE_{\phi}^{n+1}, E_{\phi}^{n+1} )_{M}=\frac{\|E_{\phi}^{n+1} \|_{M}^2- \|E_{\phi}^{n}\|_{M}^2}{2\Delta t}+
\frac{\Delta t}{2}\|d_tE_{\phi}^{n+1} \|_{M}^2.
\endaligned
\end{equation}
Using the Cauchy-Schwarz inequality, the first term on the right side of \eqref{porosity_Err7} can be bounded by
\begin{equation}\label{porosity_Err9}
\aligned
&  \chi  \left( (\frac{ 1 }{ 1+ k_s(T^{n+1} )/k_c } - 
\frac{ 1}{1+ k_s(Z^n )/k_c } ) \frac{ 1- \Psi^{n+1} }{ 1-\Psi^0 } \overline{ C }^{n}_f,  E_{\phi}^{n+1} \right)_{M} \\
\leq & C \| E_{T}^n \|_M^2 + C \| E_{\phi}^{n+1} \|_M^2 + C \|\frac{\partial T }{\partial t}\|_{L^{\infty}(J;
L^{\infty}(\Omega))}^2(\Delta t)^2 .
\endaligned
\end{equation}
Recalling Lemma \ref{le_porosity_Err1}, the second term on the right side of \eqref{porosity_Err7} can be estimated by
\begin{equation}\label{porosity_Err10}
\aligned
& \chi \left( ( 1-  \frac{ 1 }{ 1+ k_s(T^{n+1} )/k_c } )  ( \frac{ 1- \Psi^{n+1} }{ 1-\Psi^0 } - \frac{ 1-\phi^{n+1} }{ 1- \phi_{0} } )  \overline{ C }^{n}_f ,  E_{\phi}^{n+1} \right)_{M}  
\leq  C \| E_{\phi}^{n+1} \|_M^2.
\endaligned
\end{equation}
Using the Cauchy-Schwarz inequality, the third term on the right side of \eqref{porosity_Err7} can be recast as
\begin{equation}\label{porosity_Err11}
\aligned
& \chi  \left(  ( 1-  \frac{ 1 }{ 1+ k_s(T^{n+1} )/k_c } ) \frac{ 1-\phi^{n+1} }{ 1- \phi_{0} } ( \overline{ C }^{n}_f - c_{f}^{n+1} ), E_{\phi}^{n+1} \right)_{M} \\
\leq & C \| E_{c_f}^n \|_M^2 + C \| E_{\phi}^{n+1} \|_M^2 + C \|\frac{\partial c_f}{\partial t}\|_{L^{\infty}(J;
L^{\infty}(\Omega))}^2 (\Delta t)^2 .
\endaligned
\end{equation}
where we used the fact that $|\overline{C}^{n}_f -c_{f}^{n}|\leq |{C}^{n}_f -c_{f}^{n}|.$
The last term on the right side of \eqref{porosity_Err7} can be estimated by
\begin{equation}\label{porosity_Err12}
\aligned
&  ( \frac{ \partial \phi^{n+1} }{ \partial t } - d_t \phi^{n+1}, E_{\phi}^{n+1} )_{M} \leq C \| E_{\phi}^{n+1} \|_M^2 + C \|\frac{\partial^2 \phi }{\partial t^2}\|_{L^{\infty}(J;
L^{\infty}(\Omega))}^2 (\Delta t)^2. 
\endaligned
\end{equation}
Combing \eqref{porosity_Err7} with \eqref{porosity_Err8}-\eqref{porosity_Err12}, multiplying by $2\Delta t$, and summing for $n$ from 0 to $m$, $m \leq N-1$, we have
\begin{equation}\label{porosity_Err13}
\aligned
\|E_{\phi}^{m+1} \|^2_{M} \leq & \|E_{\phi}^{0} \|^2_{M} + C \Delta t \sum\limits_{n=0}^m  \|E_{c_f}^{n}\|_{M}^2+ C \Delta t \sum\limits_{n=0}^m  \|E_{T}^{n}\|_{M}^2 \\
& + C \Delta t\sum\limits_{n=0}^m \|E_{\phi}^{n+1}\|_{M}^2+ C\Delta t^2,
\endaligned
\end{equation}
which leads to the desired result \eqref{porosity_Err14}.

On the other hand, multiplying \eqref{porosity_Err6} by $ d_t E_{\phi,i,j}^{n+1} h k $, making summation on $i,j$ for $1 \leq i \leq N_x, \ 1 \leq j \leq N_y $ and following the similar procedure as \eqref{porosity_Err9}-\eqref{porosity_Err13}, we can easily obtain the desired result \eqref{porosity_Err15}. 
\end{proof}

\begin{lemma}\label{le_error_pressure}
The approximate errors of discrete pressure and velocity satisfy
\begin{equation} \label{pressure_Err1}
\aligned
&  \Delta t \sum\limits_{n=0}^{m} \|d_t E_p^{n+1} \|_{M}^2+ \|E_{ \textbf{u} }^{m+1}\|_{TM}^2 
\leq  C  \Delta t \sum\limits_{n=0}^m  (\|E_{\phi}^{n+1}\|_{M}^2+\|d_tE_{\phi}^{n+1} \|_{M}^2) \\
& \ \ \ \ \ \ \ \ \ \ \ 
+ C   \Delta t \sum\limits_{n=0}^m  \|E_{ \textbf{u} }^{n+1}\|_{TM}^2 + O( (\Delta t)^2+h^4+k^4), \ \ m\leq N-1,
\endaligned
\end{equation}
where the positive constant $C$ is independent of $h,k$ and $\Delta t$.
\end{lemma}
 
 \begin{proof}
 Since the proof of this lemma shares similar procedures with the proof of Lemma 7 in \cite{li2018block}, we omit the proof for brevity.
 \end{proof}
 
  \medskip
 \begin{lemma}\label{le_error_concentration}
The approximate error of discrete concentration satisfy
\begin{equation} \label{concentration_Err1}
\aligned
\| E_{c_f}^{m+1} \|^2_M
& + \Delta t \sum\limits_{n=0}^{m} ( \| d_x E_{c_f}^{n+1} \|_x^2 + \| d_y E_{c_f}^{n+1} \|_y^2 ) \\
\leq & C \Delta t \sum\limits_{n=0}^{m} \| E_{\phi}^{n+1} \|_M^2 + C \Delta t \sum\limits_{n=0}^{m} \| E_{c_f}^{n+1} \|_M^2 + C \Delta t \sum\limits_{n=0}^{m} \| d_t E_{\phi}^{n+1} \|_M^2  \\
& + C \Delta t \sum\limits_{n=0}^{m} \| E_{\textbf{u}}^{n+1}\|_{TM}^2 + C \Delta t \sum\limits_{n=0}^{m} \| E_{T}^n \|_M^2 + C (h^4+k^4+ (\Delta t)^2 ), \ \ m\leq N-1,
\endaligned
\end{equation}
where the positive constant $C$ is independent of $h,k$ and $\Delta t$.
\end{lemma}
 
 \begin{proof}
 Subtracting \eqref{e_Darcy3_transform} from \eqref{e_model_fully_discrete_concentration}, we have that
 \begin{equation} \label{concentration_Err2}
\aligned
 d_t (\Psi C_f - \phi c_f ) ^{n+1}_{i,j} + & [D_x E_{\textbf{w}} ]^{x,n+1}_{i,j} + [D_y E_{\textbf{w}} ]^{y,n+1}_{i,j}
=   k_c a_v( \Psi^{n+1}_{i,j} ) \left( \frac{1}{ 1+ k_s(Z^n_{i,j} ) /k_c } -1 \right) E_{c_f,i,j}^{n+1} \\
&+ k_c a_v( \Psi^{n+1}_{i,j} ) \left( \frac{1}{ 1+ k_s(Z^n_{i,j} ) /k_c } -\frac{1}{ 1+ k_s( T^{n+1}_{i,j} ) /k_c } \right) c^{n+1}_{f,i,j} \\
& + k_c ( a_v( \Psi^{n+1}_{i,j} )- a_v(\phi^{n+1}_{i,j} ) ) \frac{1}{ 1+ k_s( T^{n+1}_{i,j} ) /k_c } c^{n+1}_{f,i,j}
 + f^{n+1}_{P,i,j} E_{c_f,i,j}^{n+1} \\
 & + S^{n+1}_{1,i,j} + S^{n+1}_{2,i,j}, 
\endaligned
\end{equation}
where 
$$S_{1,i,j}^{n+1} = \dfrac{\partial (\phi c_f )_{i,j}^{n+1} }{\partial t}  - d_t (\phi c_f )_{i,j}^{n+1}, $$ 
and 
 $$S_{2,i,j}^{n+1} = [D_x w]^{x,n+1}_{i,j} + [D_y w]^{y,n+1}_{i,j} - ( \frac{ \partial w^{x,n+1}_{i,j} }{ \partial x} +  \frac{ \partial w^{y,n+1}_{i,j} }{ \partial y} ). $$
 Noting \eqref{e_model_fully_discrete_auxiliary W}, we can obtain
 \begin{equation} \label{concentration_Err3}
\aligned
 E_{\textbf{w},i+1/2,j }^{x,n+1} = & U^{x,n+1}_{i+1/2,j} \Pi_h C^{n+1}_{f,i+1/2,j} - u^{x,n+1}_{i+1/2,j} c_{f,i+1/2,j}^{n+1} \\
&
-D_{11} (  \Pi_h \Psi^{n+1}_{i+1/2,j}  [d_x C_f ]^{n+1}_{i+1/2,j} - \phi^{n+1}_{i+1/2,j} \frac{ \partial c_{f,i+1/2,j}^{n+1} }{ \partial x} ),  
\endaligned
\end{equation} 
and 
 \begin{equation} \label{concentration_Err4}
\aligned
 E_{\textbf{w},i,j+1/2 }^{y,n+1} = & U^{y,n+1}_{i,j+1/2} \Pi_h C^{n+1}_{f,i,j+1/2} - u^{y,n+1}_{i,j+1/2} c_{f,i,j+1/2}^{n+1} \\
& -D_{22} (  \Pi_h \Psi^{n+1}_{i,j+1/2}  [d_y C_f ]^{n+1}_{i,j+1/2} - \phi^{n+1}_{i,j+1/2} \frac{ \partial c^{n+1}_{f,i,j+1/2} }{ \partial y } ) . 
\endaligned
\end{equation} 
Multiplying \eqref{concentration_Err2} by $E_{c_f,i,j}^{n+1} h k $ and making summation on $i,j$ for $1 \leq i \leq N_x, \ 1 \leq j \leq N_y $ lead to  
 \begin{equation} \label{concentration_Err5}
\aligned
& \left( d_t (\Psi C_f - \phi c_f ) ^{n+1}, E_{c_f}^{n+1} \right)_M+ ( D_x E_{\textbf{w}}^{x,n+1} ,E_{c_f}^{n+1} )_M + (D_y E_{\textbf{w}}^{y,n+1}, E_{c_f}^{n+1} )_M 
\\
= & \left( k_c a_v( \Psi^{n+1} ) \left( \frac{1}{ 1+ k_s(Z^n) /k_c } -1 \right) E_{c_f}^{n+1}, 
E_{c_f}^{n+1} \right)_M \\
& + \left( k_c a_v( \Psi^{n+1} ) \left( \frac{1}{ 1+ k_s(Z^n) /k_c } -\frac{1}{ 1+ k_s( T^{n+1} ) /k_c } \right) c_f^{n+1},  E_{c_f}^{n+1} \right)_M \\
& + \left( k_c ( a_v( \Psi^{n+1} )- a_v(\phi^{n+1} ) ) \frac{1}{ 1+ k_s( T^{n+1} ) /k_c } c_f^{n+1},  E_{c_f}^{n+1} \right)_M \\
&+ ( f_P^{n+1} E_{c_f}^{n+1}, E_{c_f}^{n+1}  )_M + (S_1^{n+1}, E_{c_f}^{n+1}  )_M +
(S_2^{n+1}, E_{c_f}^{n+1}  )_M .
\endaligned
\end{equation}
Recalling Lemma \ref{le_porosity_Err1}, the first term on the left side of \eqref{concentration_Err5} can be bounded by
\begin{equation} \label{concentration_Err6}
\aligned
& \left( d_t (\Psi C_f - \phi c_f ) ^{n+1}, E_{c_f}^{n+1} \right)_M =
\left( d_t (\Psi E_{c_f}+ c_f E_{\phi} )^{n+1},   E_{c_f}^{n+1} \right)_M \\
= & ( d_t \Psi^{n+1} E_{c_f}^{n+1}, E_{c_f}^{n+1} )_M + ( \Psi^{n} d_t E_{c_f}^{n+1} , E_{c_f}^{n+1} )_M \\
& + ( d_t c_f^{n+1}E_{\phi}^{n+1}, E_{c_f}^{n+1} )_M + (c_f^{n} d_t E_{\phi}^{n+1} , E_{c_f}^{n+1} )_M \\
\geq &\frac{\phi_{0*}}{2} \frac{ \| E_{c_f}^{n+1} \|^2 - \| E_{c_f}^{n} \|^2 }{\Delta t} + 
\frac{\phi_{0*}}{2} \frac{ \| E_{c_f}^{n+1}-E_{c_f}^{n} \|^2 }{\Delta t} \\
& + ( d_t c_f^{n+1}E_{\phi}^{n+1}, E_{c_f}^{n+1} )_M  + (c_f^{n} d_t E_{\phi}^{n+1} , E_{c_f}^{n+1} )_M.
\endaligned
\end{equation}
Taking notice of \eqref{concentration_Err3} and Lemma \ref{le1}, the second term on the left side of \eqref{concentration_Err5} can be transformed into
\begin{equation} \label{concentration_Err7}
\aligned
&  ( D_x E_{\textbf{w}}^{x,n+1} ,E_{c_f}^{n+1} )_M = 
-(E_{\textbf{w}}^{x,n+1}, d_x E_{c_f}^{n+1} )_x \\
= & D_{11}  (  \Pi_h \Psi^{n+1}  [d_x C_f ]^{n+1} - \phi^{n+1} \frac{ \partial c_{f}^{n+1} }{ \partial x} , d_x E_{c_f}^{n+1} )_x  - ( U^{x,n+1} \Pi_h C^{n+1}_f - u^{x,n+1} c_{f}^{n+1}, d_x E_{c_f}^{n+1} )_x \\
= & D_{11} (  \Pi_h \Psi^{n+1} d_x E_{c_f}^{n+1} , d_x E_{c_f}^{n+1} )_x + 
D_{11} ( d_x c_{f}^{n+1} \Pi_h E_{\phi}^{n+1}, d_x E_{c_f}^{n+1} )_x \\
&+D_{11} ( \Pi_h \phi^{n+1} d_x c_{f}^{n+1} - \phi^{n+1} \frac{ \partial c_{f}^{n+1} }{ \partial x}, d_x E_{c_f}^{n+1} )_x \\
 & - ( U^{x,n+1} \Pi_h E_{c_f}^{n+1} , d_x E_{c_f}^{n+1} )_x - ( E_{\textbf{u}}^{x,n+1} \Pi_h c_f^{n+1}, d_x E_{c_f}^{n+1} )_x  \\
 &- \left( u^{x,n+1} (\Pi_h c_f^{n+1}-c_f^{n+1}), d_x E_{c_f}^{n+1} \right)_x ,
\endaligned
\end{equation}
where we shall first assume that there exists a positive constant $C^*$ such that 
\begin{equation} \label{concentration_boundedness_U}
\aligned
& \| \textbf{U}^{n+1} \|_{\infty} \leq C^*,
\endaligned
\end{equation}
and the proof of \eqref{concentration_boundedness_U} is essentially identical with the estimates in \cite{li2018block} by using an induction process. So we omit the details for simplicity.

Taking notice of \eqref{concentration_Err4} and Lemma \ref{le1}, the third term on the left side of \eqref{concentration_Err5} can be transformed into
\begin{equation} \label{concentration_Err8}
\aligned
 ( D_y E_{\textbf{w}}^{y,n+1} ,E_{c_f}^{n+1} )_M 
= & D_{22} (  \Pi_h \Psi^{n+1} d_y E_{c_f}^{n+1} , d_y E_{c_f}^{n+1} )_y + 
D_{22} ( d_y c_{f}^{n+1} \Pi_h E_{\phi}^{n+1}, d_y E_{c_f}^{n+1} )_y \\
&+D_{22} ( \Pi_h \phi^{n+1} d_y c_{f}^{n+1} - \phi^{n+1} \frac{ \partial c_{f}^{n+1} }{ \partial y}, d_y E_{c_f}^{n+1} )_y \\
 & - ( U^{y,n+1} \Pi_h E_{c_f}^{n+1} , d_y E_{c_f}^{n+1} )_y - ( E_{\textbf{u}}^{y,n+1} \Pi_h c_f^{n+1}, d_y E_{c_f}^{n+1} )_y  \\
 &- \left( u^{y,n+1} (\Pi_h c_f^{n+1}-c_f^{n+1}), d_y E_{c_f}^{n+1} \right)_y. 
\endaligned
\end{equation}
Recalling that $ a_v(\phi ) = \frac{1-\phi}{1-\phi_0} a_0 $, we have 
\begin{equation} \label{concentration_interface_a}
\aligned
 0 \leq   a_v(\phi^{n+1} ) \leq a_0.
\endaligned
\end{equation}
Thus the first term on the right side of \eqref{concentration_Err5} can be estimated by
 \begin{equation} \label{concentration_Err9}
\aligned
& \left( k_c a_v( \Psi^{n+1} ) \left( \frac{1}{ 1+ k_s(Z^n) /k_c } -1 \right) E_{c_f}^{n+1}, 
E_{c_f}^{n+1} \right)_M \leq C \| E_{c_f}^{n+1} \|_M^2.
\endaligned
\end{equation}
The second term on the right side of \eqref{concentration_Err5} can be bounded by
 \begin{equation} \label{concentration_Err10}
\aligned
&\left( k_c a_v( \Psi^{n+1} ) \left( \frac{1}{ 1+ k_s(Z^n) /k_c } -\frac{1}{ 1+ k_s( T^{n+1} ) /k_c } \right) c_f^{n+1},  E_{c_f}^{n+1} \right)_M \\
\leq & C \|T^{n+1}-Z^n \|^2_M + C \| E_{c_f}^{n+1} \|_M^2 \\
\leq & C\| E_{T}^n \|_M^2 +  C \| E_{c_f}^{n+1} \|_M^2  + C (\Delta t)^2.
\endaligned
\end{equation}
Using Cauchy-Schwartz inequality, the third term on the right side of \eqref{concentration_Err5} can be bounded by
 \begin{equation} \label{concentration_Err11}
\aligned
& \left( k_c ( a_v( \Psi^{n+1} )- a_v(\phi^{n+1} ) ) \frac{1}{ 1+ k_s( T^{n+1} ) /k_c } c_f^{n+1},  E_{c_f}^{n+1} \right)_M \\
\leq & C \| E_{\phi}^{n+1} \|_M^2 + C \| E_{c_f}^{n+1} \|_M^2 . 
\endaligned
\end{equation}
Combining \eqref{concentration_Err5} with \eqref{concentration_Err6}-\eqref{concentration_Err11} and using the Cauchy-Schwartz inequality leads to
 \begin{equation} \label{concentration_Err12}
\aligned
& \frac{\phi_{0*}}{2} \frac{ \| E_{c_f}^{n+1} \|^2_M - \| E_{c_f}^{n} \|^2_M }{\Delta t} + 
\frac{\phi_{0*}}{2} \frac{ \| E_{c_f}^{n+1}-E_{c_f}^{n} \|^2_M  }{\Delta t}  + D_{11}\phi_{0*} \| d_x E_{c_f}^{n+1} \|_x^2 + D_{22}\phi_{0*} \| d_y E_{c_f}^{n+1} \|_y^2 \\
\leq & C \| E_{\phi}^{n+1} \|_M^2 + C \| E_{c_f}^{n+1} \|_M^2 + C \| d_t E_{\phi}^{n+1} \|_M^2  + \frac{ D_{11}}{2}\phi_{0*} \| d_x E_{c_f}^{n+1} \|_x^2 \\
& + \frac{D_{22} }{2} \phi_{0*} \| d_y E_{c_f}^{n+1} \|_y^2 + C \| E_{\textbf{u}}^{n+1}\|_{TM}^2 + C\| E_{T}^n \|_M^2 + C (h^4+k^4) + C (\Delta t)^2,
\endaligned
\end{equation}
which leads to the desired result \eqref{concentration_Err1}.
 \end{proof}
 \medskip
  \begin{lemma}\label{le_error_temperature}
The approximate error of discrete temperature satisfy
\begin{equation} \label{temperature_Err1}
\aligned
&   \| E_{T}^{m+1} \|^2_M + \Delta t \sum\limits_{n=0}^m ( \| d_x E_{T}^{n+1} \|_x^2 +  \| d_y E_{T}^{n+1} \|_y^2 ) \\
\leq &  C \Delta t \sum\limits_{n=0}^m \| E_{T}^{n+1} \|_M^2 + C \Delta t \sum\limits_{n=0}^m \| E_{\phi}^{n+1} \|_M^2  + C \Delta t \sum\limits_{n=0}^m \| d_t E_{\phi}^{n+1} \|_M^2   \\
&+ C \Delta t \sum\limits_{n=0}^m \| E_{\textbf{u}}^{n+1}\|_{TM}^2  + C \Delta t \sum\limits_{n=0}^m \| E_{c_f}^{n+1} \|_M^2 + C (h^4+k^4 + C (\Delta t)^2), \ \ m\leq N-1,
\endaligned
\end{equation}
where the positive constant $C$ is independent of $h,k$ and $\Delta t$.
\end{lemma}
 
 \begin{proof}
 Subtracting \eqref{e_Heat1_transform} from \eqref{e_model_fully_discrete_Temperature}, we have that
 \begin{equation} \label{temperature_Err2}
\aligned
& \rho_s \theta_{s} d_t \left( (1-\phi)E_T+( 1-E_{\phi} ) T \right)_{i,j}^{n+1} +
\rho_f \theta_{f} d_t ( \phi E_T +E_{\phi} T )_{i,j}^{n+1} \\
& +  [D_x E_{\textbf{v}} ]^{x,n+1}_{i,j} + [D_y E_{\textbf{v}} ]^{y,n+1}_{i,j} \\
= & a_v( \Psi^{n+1}_{i,j} ) H_r(Z^n_{i,j}) R(C^{n+1}_{f,i,j}, Z^n_{i,j} ) - 
a_v( \phi^{n+1}_{i,j} ) H_{r}(T^{n+1}_{i,j} )R(c_{f,i,j}^{n+1}, T^{n+1}_{i,j} ) \\
& + S^{n+1}_{3,i,j} + S^{n+1}_{4,i,j}, 
\endaligned
\end{equation} 
where 
$$S_{3,i,j}^{n+1} = \frac{ \partial [ \left( \rho_s(1-\phi) \theta_{s} + \rho_f  \phi \theta_{f} \right) T ] }{ \partial t } |_{i,j}^{n+1}  - d_t [ \left( \rho_s(1-\phi) \theta_{s} + \rho_f  \phi \theta_{f} \right) T ]_{i,j}^{n+1}, $$ 
and 
 $$S_{4,i,j}^{n+1} = [D_x v]^{x,n+1}_{i,j} + [D_y v]^{y,n+1}_{i,j} - ( \frac{ \partial v^{x,n+1}_{i,j} }{ \partial x} +  \frac{ \partial v^{y,n+1}_{i,j} }{ \partial y} ). $$
 Noting \eqref{e_model_fully_discrete_auxiliary V}, we can obtain
 \begin{equation} \label{temperature_Err3}
\aligned
 E_{\textbf{v},i+1/2,j }^{x,n+1} = & \rho_f \theta_{f} U^{x,n+1}_{i+1/2,j} \Pi_h Z^{n+1}_{i+1/2,j} -  \rho_f \theta_{f} u^{x,n+1}_{i+1/2,j}  T^{n+1}_{i+1/2,j} \\
 & - \left( \lambda ( \Pi_h \Psi^{n+1}_{i+1/2,j} ) [d_x Z]^{n+1}_{i+1/2,j} - 
  \lambda(\phi^{n+1}_{i+1/2,j} ) \frac{ \partial T_{i+1/2,j}^{n+1} }{ \partial x}
 \right),  
\endaligned
\end{equation} 
and 
 \begin{equation} \label{temperature_Err4}
\aligned
 E_{\textbf{v},i,j+1/2 }^{y,n+1} = & \rho_f \theta_{f} U^{y,n+1}_{i,j+1/2} \Pi_h Z^{n+1}_{i,j+1/2} -  \rho_f \theta_{f} u^{y,n+1}_{i,j+1/2}  T^{n+1}_{i,j+1/2} \\
 & - \left( \lambda ( \Pi_h \Psi^{n+1}_{i,j+1/2 } ) [d_y Z]^{n+1}_{i,j+1/2 } - 
  \lambda(\phi^{n+1}_{i,j+1/2 } ) \frac{ \partial T_{i,j+1/2 }^{n+1} }{ \partial y}
 \right).  
\endaligned
\end{equation} 
Multiplying \eqref{temperature_Err2} by $E_{T,i,j}^{n+1} h k $ and making summation on $i,j$ for $1 \leq i \leq N_x, \ 1 \leq j \leq N_y $ lead to  
\begin{equation} \label{temperature_Err5}
\aligned
& \rho_s \theta_{s} \left( d_t \left( (1-\Psi)E_T - E_{\phi}  T \right)^{n+1} , E_{T}^{n+1} \right)_M + \rho_f \theta_{f} \left( d_t ( \Psi E_T +E_{\phi} T )^{n+1}, E_{T}^{n+1} 
\right)_M \\
&+ ( D_x E_{\textbf{v}}^{x,n+1} ,E_{T}^{n+1} )_M + (D_y E_{\textbf{v}}^{y,n+1}, E_{T}^{n+1} )_M \\
= & \left( a_v( \Psi^{n+1} ) H_r(Z^n ) R( \overline{C}^{n+1}_f, Z^n ) - 
a_v( \phi^{n+1} ) H_{r}(T^{n+1} )R(c_{f}^{n+1}, T^{n+1} ), E_{T}^{n+1} 
\right)_M \\
& + ( S^{n+1}_{3} + S^{n+1}_{4}, E_{T}^{n+1} )_M,
\endaligned
\end{equation} 
The first two terms on the left side of \eqref{temperature_Err5} can be transformed into 
\begin{equation} \label{temperature_Err6}
\aligned
& \rho_s \theta_{s} \left( d_t \left( (1-\Psi)E_T -E_{\phi}  T \right)^{n+1} , E_{T}^{n+1} \right)_M + \rho_f \theta_{f} \left( d_t ( \Psi E_T +E_{\phi} T )^{n+1}, E_{T}^{n+1} 
\right)_M \\
= & \rho_s \theta_{s} \left( E_T^{n+1} d_t (1-\Psi^{n+1} ) + (1-\Psi^{n} ) d_t E_{T}^{n+1} , E_{T}^{n+1} 
\right)_M \\
& - \rho_s \theta_{s} \left( T^{n+1} d_t E_{\phi}^{n+1}  + E_{\phi}^{n}  d_t T^{n+1} , E_{T}^{n+1} \right)_M \\
& +  \rho_f \theta_{f} \left( d_t \Psi^{n+1} E_T^{n+1} + \Psi^n d_t E_T^{n+1} + T^{n+1} d_t E_{\phi}^{n+1} + E_{\phi}^{n} d_t T^{n+1}  , E_{T}^{n+1} \right)_M \\
=& \big( \left(  \rho_s \theta_{s} ( 1- \Psi^n) +  \rho_f \theta_{f} \Psi^n \right) d_tE_T^{n+1}  , E_{T}^{n+1} \big)_M \\
& + \big( \left(  \rho_s \theta_{s} d_t ( 1- \Psi^{n+1} ) +  \rho_f \theta_{f} d_t \Psi^{n+1} \right) E_T^{n+1}  , E_{T}^{n+1} \big)_M \\
& + \left(   ( \rho_f \theta_{f} - \rho_s \theta_{s} ) ( d_t E_{\phi}^{n+1} T^{n+1} + E_{\phi}^{n} d_t T^{n+1} ) , E_{T}^{n+1} \right)_M  ,
\endaligned
\end{equation}
where we should note that $\rho_s \theta_{s} ( 1- \Psi^n_{i,j} ) +  \rho_f \theta_{f} \Psi^n_{i,j} \geq \min\{ \rho_s \theta_{s}, \rho_f \theta_{f} \}$ due to \eqref{porosity_Err2}.

Taking notice of \eqref{temperature_Err3} and Lemma \ref{le1}, the third term on the left side of \eqref{temperature_Err5} can be transformed into
\begin{equation} \label{temperature_Err7}
\aligned
& ( D_x E_{\textbf{v}}^{x,n+1} ,E_{T}^{n+1} )_M  = 
-(E_{\textbf{v}}^{x,n+1}, d_x E_{T}^{n+1} )_x \\
= &   \left(  \lambda ( \Pi_h \Psi^{n+1} )  [d_x Z]^{n+1} -  \lambda(\phi^{n+1} ) \frac{ \partial T^{n+1} }{ \partial x} , d_x E_{T}^{n+1} \right)_x  \\
& - \rho_f \theta_{f} ( U^{x,n+1} \Pi_h Z^{n+1} - u^{x,n+1} T^{n+1}, d_x E_T^{n+1} )_x \\
=  &  \left(  \lambda ( \Pi_h \Psi^{n+1} ) d_x E_{T}^{n+1} , d_x E_{T}^{n+1} \right)_x + 
 \left( d_x T^{n+1} \lambda ( \Pi_h E_{\phi}^{n+1} ), d_x E_{T}^{n+1} \right)_x \\
&+\left(  \lambda (  \Pi_h \phi^{n+1} ) d_x T^{n+1} -  \lambda ( \phi^{n+1} ) \frac{ \partial T^{n+1} }{ \partial x}, d_x E_{T}^{n+1} \right)_x \\
 & - \rho_f \theta_{f} ( U^{x,n+1} \Pi_h E_{T}^{n+1} , d_x E_{T}^{n+1} )_x - \rho_f \theta_{f}  ( E_{\textbf{u}}^{x,n+1} \Pi_h T^{n+1}, d_x E_{T}^{n+1} )_x  \\
 &- \rho_f \theta_{f}  \left( u^{x,n+1} (\Pi_h T^{n+1}-T^{n+1}), d_x E_{T}^{n+1} \right)_x ,
\endaligned
\end{equation}  
where we should note that $ \lambda ( \Pi_h \Psi^{n+1} ) \geq \min\{ \lambda_{s}, \lambda_f \}$ due to \eqref{porosity_Err2}.

Taking notice of \eqref{temperature_Err4} and Lemma \ref{le1}, the last term on the left side of \eqref{temperature_Err5} can be transformed into
\begin{equation} \label{temperature_Err8}
\aligned
& ( D_y E_{\textbf{v}}^{y,n+1} ,E_{T}^{n+1} )_M  = 
-(E_{\textbf{v}}^{y,n+1}, d_y E_{T}^{n+1} )_y \\
= &   \left(  \lambda ( \Pi_h \Psi^{n+1} )  [d_y Z]^{n+1} -  \lambda(\phi^{n+1} ) \frac{ \partial T^{n+1} }{ \partial y} , d_y E_{T}^{n+1} \right)_y  \\
& - \rho_f \theta_{f} ( U^{y,n+1} \Pi_h Z^{n+1} - u^{y,n+1} T^{n+1}, d_y E_T^{n+1} )_y \\
=  &  \left(  \lambda ( \Pi_h \Psi^{n+1} ) d_y E_{T}^{n+1} , d_y E_{T}^{n+1} \right)_y + 
 \left( d_y T^{n+1} \lambda ( \Pi_h E_{\phi}^{n+1} ), d_y E_{T}^{n+1} \right)_y \\
&+\left(  \lambda (  \Pi_h \phi^{n+1} ) d_y T^{n+1} -  \lambda ( \phi^{n+1} ) \frac{ \partial T^{n+1} }{ \partial y}, d_y E_{T}^{n+1} \right)_y \\
 & - \rho_f \theta_{f} ( U^{y,n+1} \Pi_h E_{T}^{n+1} , d_y E_{T}^{n+1} )_y - \rho_f \theta_{f}  ( E_{\textbf{u}}^{y,n+1} \Pi_h T^{n+1}, d_y E_{T}^{n+1} )_y  \\
 &- \rho_f \theta_{f}  \left( u^{y,n+1} (\Pi_h T^{n+1}-T^{n+1}), d_y E_{T}^{n+1} \right)_y .
\endaligned
\end{equation} 
Recalling \eqref{e_Darcy5} and \eqref{e_Darcy6}, we have 
 \begin{equation} \label{temperature_R}
\aligned
& R(c_{f}, T ) = k_c \left( 1- \frac{1}{ 1+ k_s(T) /k_c } \right) c_f,
\endaligned
\end{equation}  
Thus taking notice of \eqref{concentration_interface_a} and using Cauchy-Schwartz inequality, the first term on the right side of \eqref{temperature_Err5} can be estimated by
 \begin{equation} \label{temperature_Err9}
\aligned
& \left( a_v( \Psi^{n+1} ) H_r(Z^n ) R( \overline{C}^{n+1}_f, Z^n ) - 
a_v( \phi^{n+1} ) H_{r}(T^{n+1} )R(c_{f}^{n+1}, T^{n+1} ), E_{T}^{n+1} 
\right)_M  \\ 
= & \big( a_v( \Psi^{n+1} )  R( \overline{C}^{n+1}_f, Z^n ) \left( H_r(Z^n) - H_r(T^{n+1} ) \right), E_{T}^{n+1} \big)_M \\
& + \Big( a_v( \Psi^{n+1} )  H_r(T^{n+1} ) \big( R( \overline{C}^{n+1}_f, Z^n )- R(c_{f}^{n+1}, T^{n+1}  ) \big), E_{T}^{n+1} \Big)_M \\
& + \big( R(c_{f}^{n+1}, T^{n+1} ) H_r(T^{n+1} )  \left( a_v( \Psi^{n+1} )  -  a_v( \phi^{n+1} )  \right), E_{T}^{n+1} \big)_M \\
\leq &  C \| E_T^{n+1} \|^2_M + C \| E_T^{n} \|^2_M  + C \| E_{c_f}^{n+1} \|_M^2 
 + C \| E_{\phi}^{n+1} \|_M^2 
+ C(\Delta t)^2 .
\endaligned
\end{equation}  
Combining \eqref{temperature_Err5} with \eqref{temperature_Err6}-\eqref{temperature_Err9} and using the Cauchy-Schwartz inequality leads to
 \begin{equation} \label{temperature_Err10}
\aligned
&  \frac{ \| E_{T}^{n+1} \|^2_M - \| E_{T}^{n} \|^2_M }{\Delta t} + 
 \frac{ \| E_{T}^{n+1}-E_{T}^{n} \|^2_M }{\Delta t}  +  \| d_x E_{T}^{n+1} \|_x^2 +  \| d_y E_{T}^{n+1} \|_y^2 \\
\leq &  C \| E_{T}^{n+1} \|_M^2 + C \| E_{\phi}^{n+1} \|_M^2  + C \| d_t E_{\phi}^{n+1} \|_M^2  + C \| E_{\textbf{u}}^{n+1}\|_{TM}^2  \\
& + C\| E_{T}^n \|_M^2 + C \| E_{c_f}^{n+1} \|_M^2 + C (h^4+k^4) + C (\Delta t)^2,
\endaligned
\end{equation}
which leads to the desired result \eqref{temperature_Err1}. 
  \end{proof}
  
We are now in position to derive our main results. 
 
  \medskip
\begin{theorem}\label{thm_main}
 Suppose the analytical solutions are sufficiently smooth, then for the fully-discrete scheme \eqref{e_model_fully_discrete_Darcy}-\eqref{e_model_fully_discrete_porosity}, we have
\begin{equation} \label{main_result}
\aligned
& \|E_{\phi}^{m+1} \|^2_{M} + \|E_{ \textbf{u} }^{m+1}\|_{TM}^2 + \|E_p^{m+1} \|_{M}^2 + \| E_{c_f}^{m+1} \|^2_M +  \| E_{T}^{m+1} \|^2_M \\
\leq  &   C ( (\Delta t)^2+h^4+k^4), \ \ m\leq N-1,
\endaligned
\end{equation} 
where the positive constant $C$ is independent of $h,k$ and $\Delta t$.
\end{theorem}

\begin{proof}
Combining Lemmas \ref{le_error_porosity}-\ref{le_error_temperature} leads to
\begin{equation} \label{main_Err1}
\aligned
& \|E_{\phi}^{m+1} \|^2_{M} + \|E_{ \textbf{u} }^{m+1}\|_{TM}^2 + \Delta t \sum\limits_{n=0}^{m} \|d_t E_p^{n+1} \|_{M}^2 + \| E_{c_f}^{m+1} \|^2_M \\
&  + \Delta t \sum\limits_{n=0}^{m} ( \| d_x E_{c_f}^{n+1} \|_x^2 + \| d_y E_{c_f}^{n+1} \|_y^2 )+  \| E_{T}^{m+1} \|^2_M  \\
&  + \Delta t \sum\limits_{n=0}^m ( \| d_x E_{T}^{n+1} \|_x^2 +  \| d_y E_{T}^{n+1} \|_y^2 ) \\
\leq  &  C \Delta t \sum\limits_{n=0}^m  \|E_{c_f}^{n+1}\|_{M}^2+ C \Delta t \sum\limits_{n=0}^m  \|E_{T}^{n+1} \|_{M}^2 +  C \Delta t\sum\limits_{n=0}^m \|E_{\phi}^{n+1}\|_{M}^2 \\
&+ C \Delta t \sum\limits_{n=0}^{m} \| E_{\textbf{u}}^{n+1}\|_{TM}^2  + O( (\Delta t)^2+h^4+k^4), \ \ m\leq N-1.
\endaligned
\end{equation}
Then supposing $\Delta t$ sufficiently small and applying the discrete Gronwall inequality, we have
\begin{equation} \label{main_Err2}
\aligned
& \|E_{\phi}^{m+1} \|^2_{M} + \|E_{ \textbf{u} }^{m+1}\|_{TM}^2 + \Delta t \sum\limits_{n=0}^{m} \|d_t E_p^{n+1} \|_{M}^2 + \| E_{c_f}^{m+1} \|^2_M \\
&  + \Delta t \sum\limits_{n=0}^{m} ( \| d_x E_{c_f}^{n+1} \|_x^2 + \| d_y E_{c_f}^{n+1} \|_y^2 )+  \| E_{T}^{m+1} \|^2_M  \\
&  + \Delta t \sum\limits_{n=0}^m ( \| d_x E_{T}^{n+1} \|_x^2 
+  \| d_y E_{T}^{n+1} \|_y^2 ) \\
\leq  &   C ( (\Delta t)^2+h^4+k^4), \ \ m\leq N-1.
\endaligned
\end{equation}
Furthermore, noting
\begin{equation} \label{main_Err3}
\aligned
 E_p^{m+1}=\Delta t \sum\limits_{n=0}^{m} d_t E_p^{n+1} + E_p^{0},
\endaligned
\end{equation}
we have
\begin{equation} \label{main_Err4}
\aligned
\|E_p^{m+1} \|_{M}^2 \leq 2T\Delta t \sum\limits_{n=0}^{m} \| d_t E_p^{n+1} \|_M^2 + 2 \|E_p^{0} \|_{M}^2 \leq C ( (\Delta t)^2+h^4+k^4),
\endaligned
\end{equation}
which leads to the desired result \eqref{main_result}.

\end{proof}

\section{Numerical simulation}
In this section we provide some two- and three-dimensional numerical experiments to gauge the constructed scheme \eqref{e_model_fully_discrete_Darcy}-\eqref{e_model_fully_discrete_auxiliary V} and \eqref{e_model_fully_discrete_porosity}.
In our simulation, we set 
\begin{equation}\label{e_Darcy7}
k_s=k_{s0}  \exp\left( \frac{E_g}{R_g}(\frac{1}{T_0}- \frac{1}{T} ) \right),
\end{equation}
where $k_{s0}$ is the surface reaction rate at temperature $T_0$, $E_g$ is the activation energy, and $R_g$ is the molar gas constant \cite{kalia2010fluid,wu2021thermodynamically}. The reaction heat 
$$ H_r(T)=| -9702 +16.97 T - 0.00234 T^2 | $$ 
is generated by reaction when per unit mole of acid is consumed \cite{li2017simulation,wu2021thermodynamically}.


\subsection{Convergence rates for the wormhole model with heat transmission process in 2- and 3-D cases}

In this subsection, the domain $\Omega=(0,1)^d$ and J=[0,1], 
 $ \Delta t=h^2 $.
 We set the following parameters:
\begin{equation}
    \begin{aligned}
        &\alpha = 1;   \rho_s = 10; a_0 = 1;   k_c = 1;   
        k_{s0} = 1;
        E_g = 1;    R_g = 1; 
        \gamma = 1;   \mu =1;\\
       &D = 1E-2;
        \rho_f = 1; \theta_s = 1;    \theta_f = 10;
        \lambda_s = 10; \lambda_f = 1;
        C_I = 1,
    \end{aligned}
\end{equation}
and test the following system to verify the convergence rates
\begin{numcases}{}
\dfrac{\partial (\phi c_f)}{\partial t}+\nabla\cdot (\textbf{u}c_f)
-\nabla\cdot (\phi\textbf{D} \nabla c_f)=k_ca_v(c_s-c_f)+f_Pc_f+f_Ic_I+ g,
 \label{e_Darcy3_Convergence} \\
\frac{\partial \phi}{\partial t}=\frac{ R(c_f,T) a_v \alpha }{\rho_s} + h, \label{e_Darcy4_Convergence} \\
\frac{ \partial [ \left( \rho_s(1-\phi) \theta_{s} + \rho_f  \phi \theta_{f} \right) T ]}{ \partial t } +  \nabla \cdot (\rho_f \theta_{f} \textbf{u}  T) = \nabla \cdot ( \lambda \nabla T ) + a_v(\phi) H_r(T)R(c_f, T)+ q, \label{e_Darcy6_Convergence}
\end{numcases} 
where $g$, $h$ and $q$ are three introduced functions to satisfy the analytic solutions given in the following two examples.

\textbf{Example 1 in 2-D case}: 
Here the initial condition and the right hand side of the equation are computed according to the analytic solution given as below:
\begin{equation*}
  \left\{
   \begin{array}{l}
    p(\textbf{x},t)=t x^2(1-x)^2y^2(1-y)^2 + 1,\\
   c_f(\textbf{x},t)=1 + tcos(\pi x)cos(\pi y),\\
    T(\textbf{x},t)=\frac{1}{2}tcos(\pi x)cos(\pi y) + 10,\\
   \phi(\textbf{x},t)= \frac{1}{4}tx^2(1-x)^2sin(\pi y) + \frac{1}{4}.
   \end{array}
   \right.
  \end{equation*}
The numerical results are listed in Tables \ref{table1_example1}-\ref{table2_example1} and give solid supporting evidence for the expected second-order convergence  of the constructed scheme in 2-D case for the wormhole model, which are consistent with the error estimates in Theorem \ref{thm_main}.

\begin{table}[htbp]
\renewcommand{\arraystretch}{1.1}
\small
\centering
\caption{Errors and convergence rates for Example 1 in 2-D case}\label{table1_example1}
\begin{tabular}{ccccccccccc} 
 \hline
$h$    &$\| E_{\phi} \|_{l^{\infty}(M)}$    &Rate  &$\| E_{p }\|_{l^{\infty}(M)}$    &Rate &$\| E_{\textbf{u} }\|_{l^{\infty}(TM)}$   &Rate    \\ \hline
1/10  &  2.49E-4  & ---   &  2.92E-4  &---    &  8.53E-5  &---     \\ 
1/20  &  6.22E-5  &  2.00  &  7.26E-5  &  2.01  &  2.12E-5  &  2.01    \\ 
1/40  &  1.55E-5  &  2.00  &  1.81E-5  &  2.00  &  5.30E-6  &  2.00   \\ 
1/80  &  3.89E-6  &  2.00  &  4.54E-6  &  2.00  &  1.32E-6  &  2.00    \\ 
\hline
\end{tabular}
\end{table} 

\begin{table}[htbp]
\renewcommand{\arraystretch}{1.1}
\small
\centering
\caption{Errors and convergence rates for Example 1 in 2-D case}\label{table2_example1}
\begin{tabular}{ccccccccccc} 
 \hline
$h$      &$\| E_{c_f }\|_{l^{\infty}(M)}$   &Rate &$\| E_{T }\|_{l^{\infty}(M)}$   &Rate  \\ \hline
1/10     &  3.78E-4  &    --- &  1.89E-1
  & ---   \\ 
1/20   &  9.42E-5  &  2.01  &  4.76E-2  &  1.99  \\ 
1/40   &  2.35E-5  &  2.00  &  1.19E-2  &  2.00  \\ 
1/80    &  5.88E-6  &  2.00  &  2.98E-3  &  2.00  \\ 
\hline
\end{tabular}
\end{table} 

\textbf{Example 2 in 3-D case}: 
Here the initial condition and the right hand side of the equation are computed according to the analytic solution given as below:
\begin{equation*}
  \left\{
   \begin{array}{l}
    p(\textbf{x},t)=(e^t-1) x^4(1-x)^4cos(\pi y)cos(\pi z) + 1,\\
   c_f(\textbf{x},t)=1 + tx^3(1-x)^3cos(\pi y)cos(\pi z),\\
    T(\textbf{x},t)=\frac{1}{2}(e^t-1)cos(\pi x)cos(\pi y)cos(\pi z) + 10,\\
   \phi(\textbf{x},t)= \frac{1}{4}(e^t-1)cos(\pi x)sin(\pi y)cos(\pi z) + \frac{1}{2}.
   \end{array}
   \right.
  \end{equation*}
The numerical results are listed in Tables \ref{table1_example2}-\ref{table2_example2} and give solid supporting evidence for the expected second-order convergence  of the constructed scheme in 3-D case for the wormhole model, which are consistent with the error estimates in Theorem \ref{thm_main}.

\begin{table}[htbp]
\renewcommand{\arraystretch}{1.1}
\small
\centering
\caption{Errors and convergence rates for Example 2 in 3-D case}\label{table1_example2}
\begin{tabular}{ccccccccccc} 
 \hline
$h$    &$\| E_{\phi} \|_{l^{\infty}(M)}$    &Rate  &$\| E_{p }\|_{l^{\infty}(M)}$    &Rate &$\| E_{\textbf{u} }\|_{l^{\infty}(TM)}$   &Rate    \\ \hline
1/10  &  8.82E-4  & ---   &  4.76E-5  & ---   &  1.61E-3  & ---    \\ 
1/20  &  2.20E-4  &  2.00  &  1.17E-5  &  2.02  &  4.56E-4  &  1.82    \\ 
1/40  &  5.50E-5  &  2.00  &  2.92E-6  &  2.01  &  1.13E-4  &  2.01   \\ 
\hline
\end{tabular}
\end{table} 

\begin{table}[htbp]
\renewcommand{\arraystretch}{1.1}
\small
\centering
\caption{Errors and convergence rates for Example 2 in 3-D case}\label{table2_example2}
\begin{tabular}{ccccccccccc} 
 \hline
$h$      &$\| E_{c_f }\|_{l^{\infty}(M)}$   &Rate &$\| E_{T }\|_{l^{\infty}(M)}$   &Rate  \\ \hline
1/10    &  3.59E-4  & ---   &  6.97E-2  &  ---   \\ 
1/20    &  8.98E-5  &  2.00  &  1.74E-2  &  2.00  \\ 
1/40    &  2.25E-5  &  2.00  &  4.35E-3  &  2.00  \\ 
\hline
\end{tabular}
\end{table} 
 
 \subsection{Simulation of dissolution patterns}
In the following examples, we set $ \Omega=(0,0.2)^d.$
Here we give the more realistic physical parameter
 \begin{equation*}
    \begin{aligned}
        &\alpha = 5E-2;   \rho_s = 2.71E3; a_0 = 5.0E-1;   k_c = 1E-3; k_{s0} = 2E-3;\\
        &E_g = 5.02416E4;    R_g = 8.314; 
        \gamma = 1E0;   \mu =1.0E-3;\\
        & D = 1E-9; \rho_f = 1.01E3; \theta_s = 2.0E2;    \theta_f = 4.184E3;
        \lambda_s = 5.526; \lambda_f = 5.8E-1;C_I = 1E3.
    \end{aligned}
\end{equation*}
Initial conditions are given as below:
\begin{equation*}
    T_0 = 2.98E2;   p_0 = 1.52E5;  c_{f0}=0  .
\end{equation*}

\textbf{Example 3 with Neumann boundary condition for temperature in 2-D case}: 
In this example, we set $ J=[0,1\times10^7],\ \Delta t=1\times10^5$ s. The distributions of initial porosity and permeability in 2-D case are listed as follows:
\begin{equation*}
   \left\{
   \begin{array}{l}
\phi_0=0.5,\ \ K_0=10^{-7},\ \ \ \ (x,y)=(1.25E-3,1.0125E-1),\\
\phi_0=0.6,\ \ K_0=10^{-6},\ \ \ \ (x,y)=(1.25E-3,5.125E-2),\\
\phi_0=0.2,\ \ K_0=10^{-8},\ \ \ \ otherwise.\\
   \end{array}
   \right.
  \end{equation*}
We set following right hand side
    \begin{flalign*}
f_I=\left\{
   \begin{array}{l}
 1E-4 \ \rm{m/s},\ \  x=1.25E-3 ,\\
  0,\ \ otherwise.\\
   \end{array}
   \right.
f_P=\left\{
   \begin{array}{l}
  -1E-4 \ \rm{m/s},\ \  x=1.9875E-1,\\
  0,\ \ otherwise.\\
   \end{array}
   \right.
\end{flalign*}

 \begin{figure}[!htp]
\centering
\includegraphics[scale=0.45]{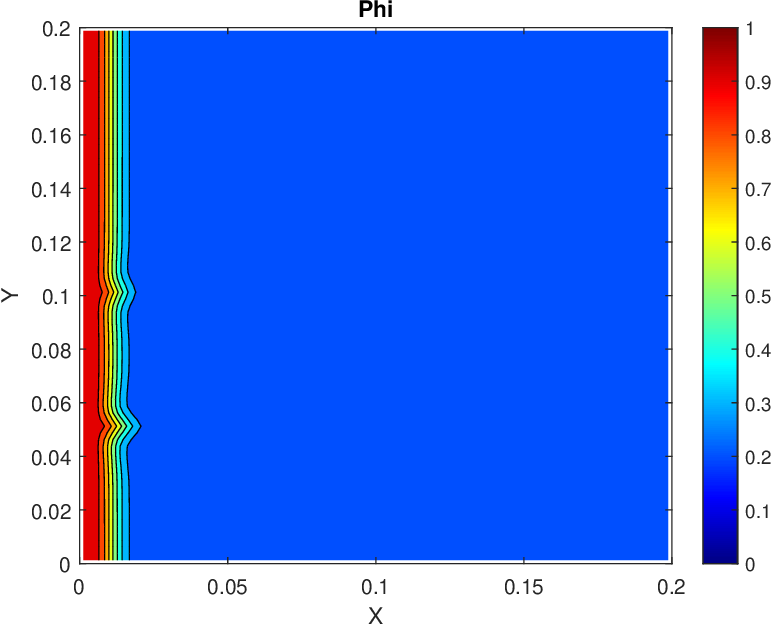}
\includegraphics[scale=0.45]{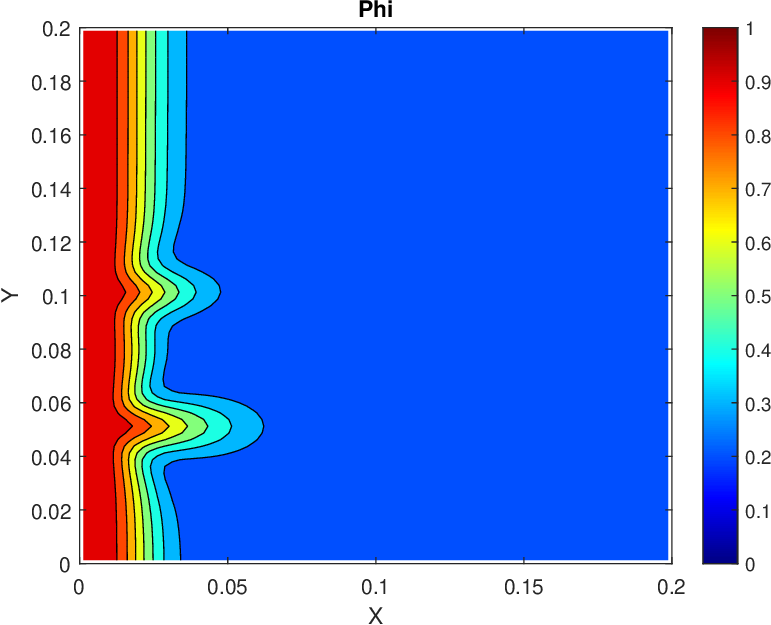}
\includegraphics[scale=0.45]{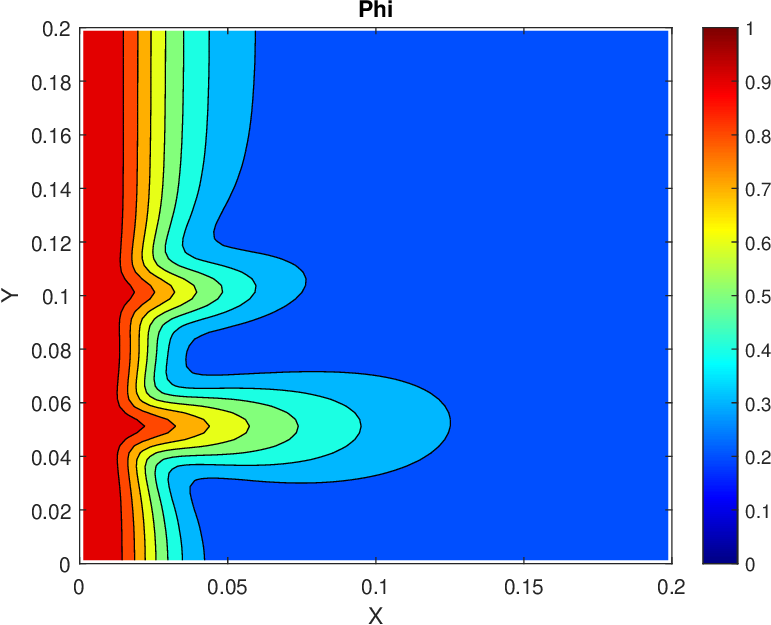}
\includegraphics[scale=0.45]{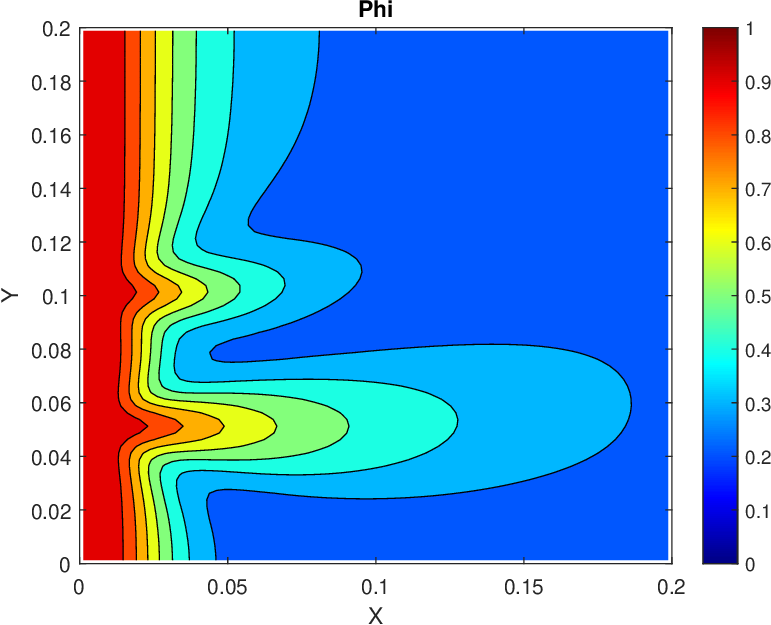}
\caption{ The distributions of porosity for Example 3 with Neumann boundary condition for temperature.}  \label{fig: porosity_example 3}
\end{figure}

\textbf{Example 4 with Robin boundary condition for temperature  in 2-D case}: In this example, we set that the temperature is 298K on the left side and satisfies homogenous Neumann condition on the other boundaries. Here $ J=[0,1\times10^6],\ \Delta t=1\times10^4$ s. 

In this example, the distributions of initial porosity and permeability in 2-D case are listed as follows:
\begin{equation*}
   \left\{
   \begin{array}{l}
\phi_0=0.5,\ \ K_0=10^{-7},\ \ \ \ (x,y)=(1.25E-3,1.0125E-1),\\
\phi_0=0.6,\ \ K_0=10^{-6},\ \ \ \ (x,y)=(1.25E-3,5.125E-2),\\
\phi_0=0.2,\ \ K_0=10^{-8},\ \ \ \ otherwise.\\
   \end{array}
   \right.
  \end{equation*}
We set following right hand side
    \begin{flalign*}
f_I=\left\{
   \begin{array}{l}
 5E-4 \ \rm{m/s},\ \  x=1.25E-3 ,\\
  0,\ \ otherwise.\\
   \end{array}
   \right.
f_P=\left\{
   \begin{array}{l}
  -5E-4 \ \rm{m/s},\ \  x=1.9875E-1,\\
  0,\ \ otherwise.\\
   \end{array}
   \right.
\end{flalign*}

 \begin{figure}[!htp]
\centering
\includegraphics[scale=0.45]{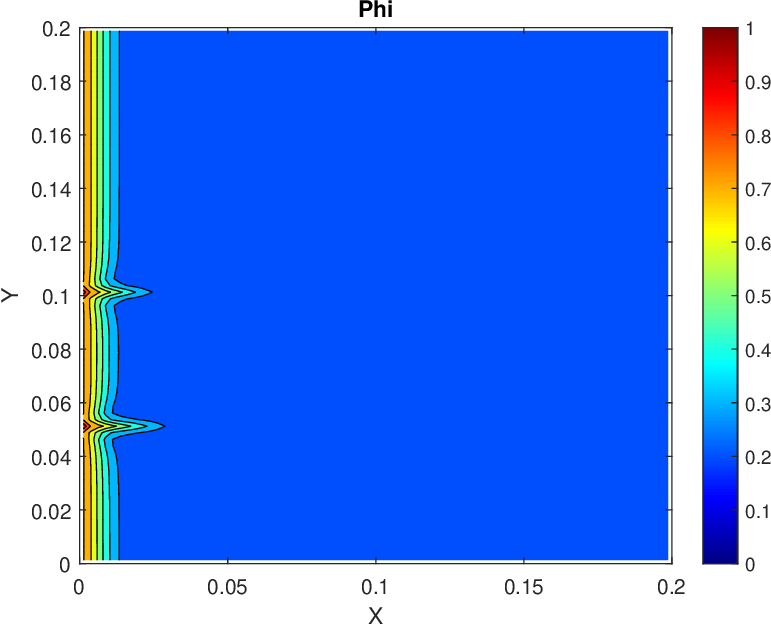}
\includegraphics[scale=0.45]{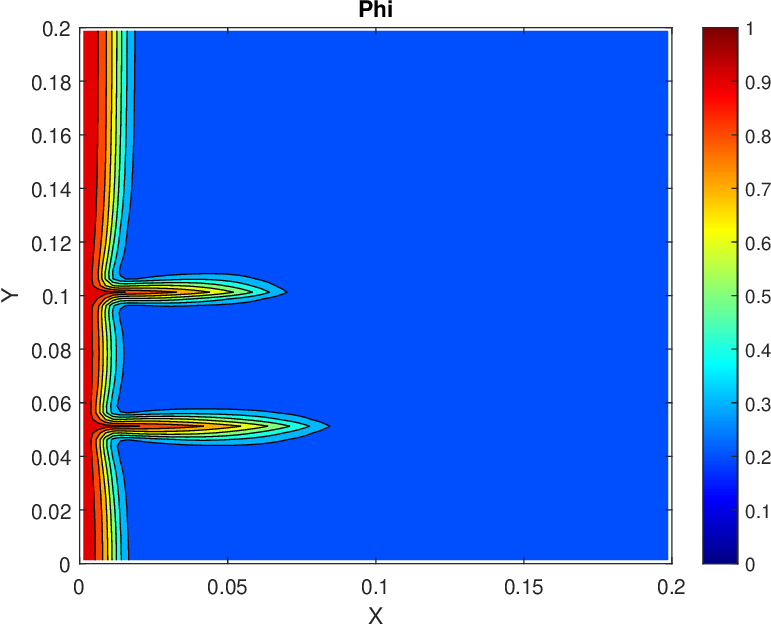}
\includegraphics[scale=0.45]{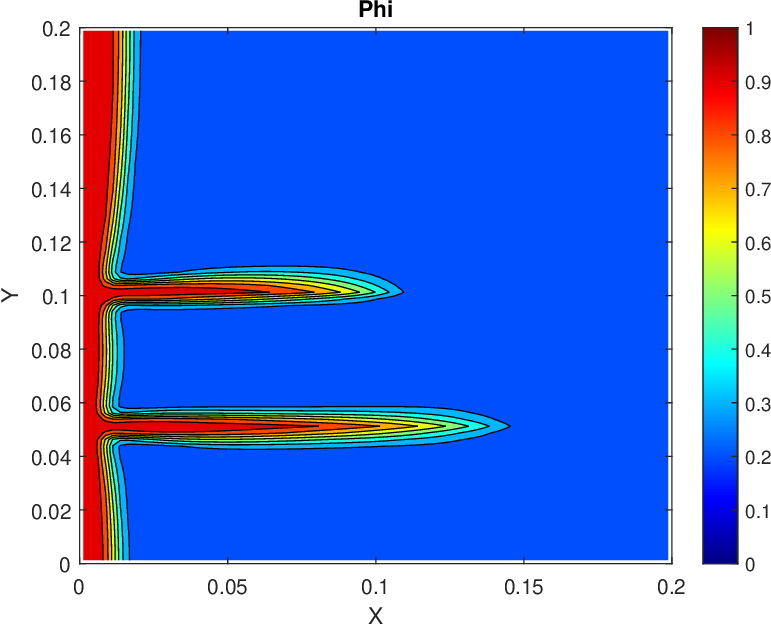}
\includegraphics[scale=0.45]{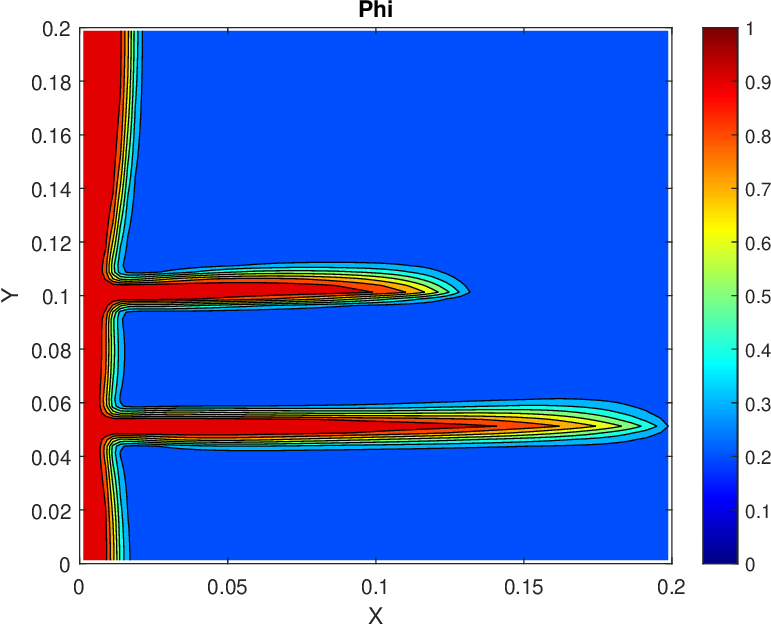}
\caption{ The distributions of porosity for Example 4 with Robin boundary condition for temperature.}  \label{fig: porosity_example 4}
\end{figure}

\textbf{Example 5 with Neumann boundary condition for temperature in 3-D case}: 
In this example, we set $ J=[0,1\times10^6],\ \Delta t=1\times10^4$ s. The distributions of initial porosity and permeability in 3-D case are listed as follows:
\begin{equation*}
   \left\{
   \begin{array}{l}
\phi_0=0.5,\ \ K_0=10^{-7},\ \ \ \ (x,y)=(2.50E-3,1.025E-1,1.025E-1),\\
\phi_0=0.6,\ \ K_0=10^{-6},\ \ \ \ (x,y)=(2.50E-3,5.25E-2,5.25E-2),\\
\phi_0=0.2,\ \ K_0=10^{-8},\ \ \ \ otherwise.\\
   \end{array}
   \right.
  \end{equation*}
We set the following right hand side
    \begin{flalign*}
f_I=\left\{
   \begin{array}{l}
 1E-4 \ \rm{m/s},\ \  x=2.50E-3 ,\\
  0,\ \ otherwise.\\
   \end{array}
   \right.
f_P=\left\{
   \begin{array}{l}
  -1E-4 \ \rm{m/s},\ \  x=1.975E-1,\\
  0,\ \ otherwise.\\
   \end{array}
   \right.
\end{flalign*}

 \begin{figure}[!htp]
\centering
\includegraphics[scale=0.45]{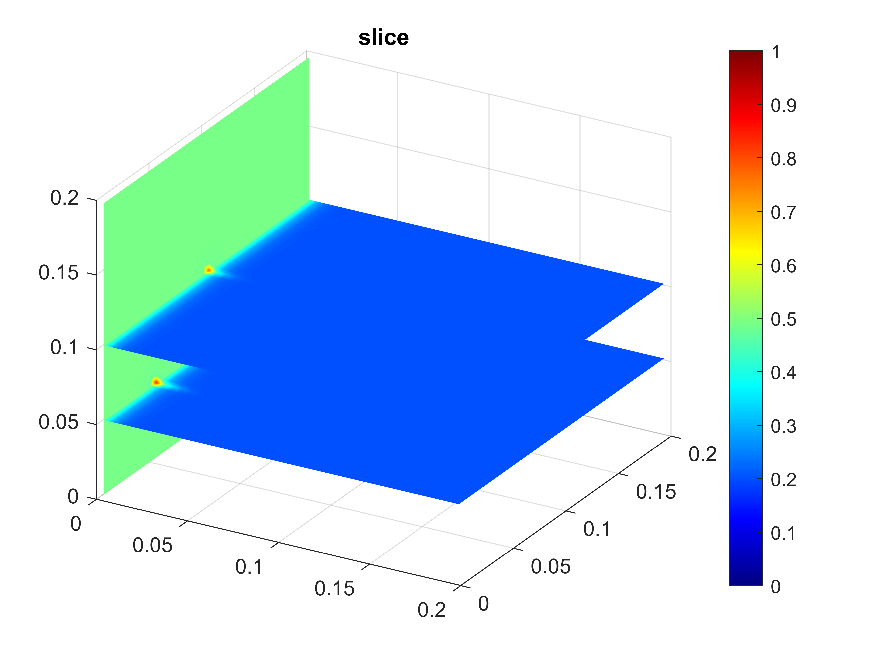}
\includegraphics[scale=0.45]{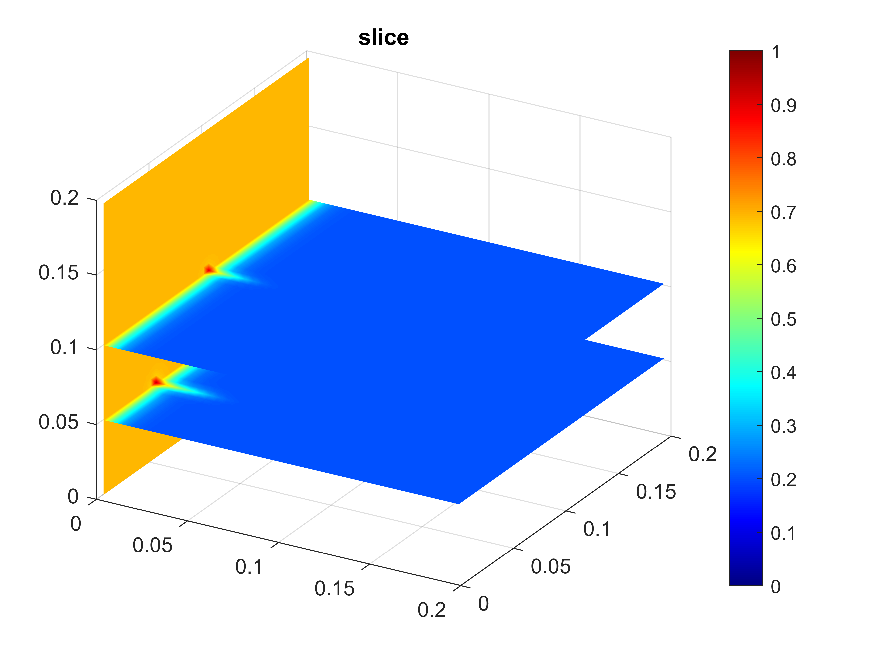}
\includegraphics[scale=0.45]{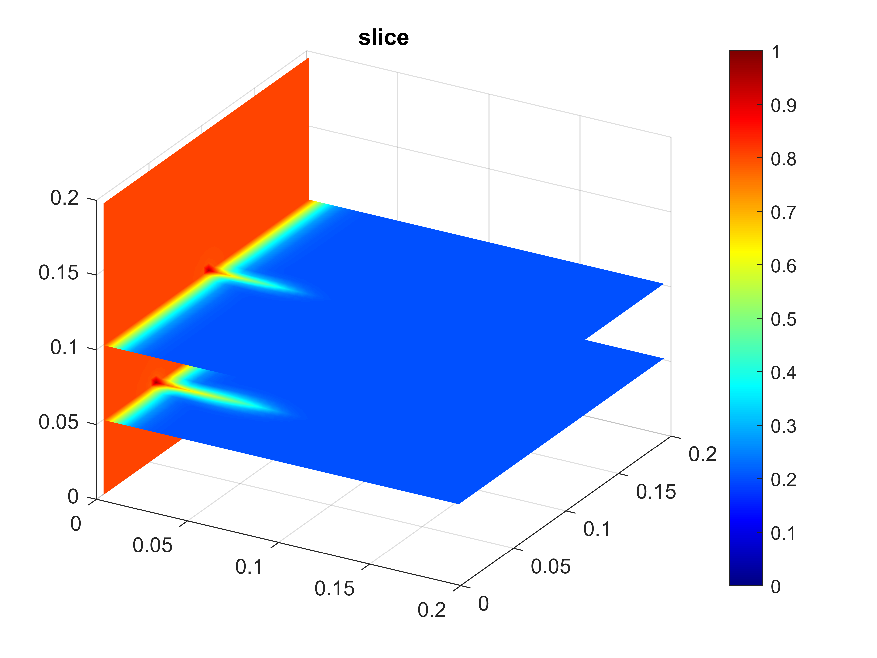}
\includegraphics[scale=0.45]{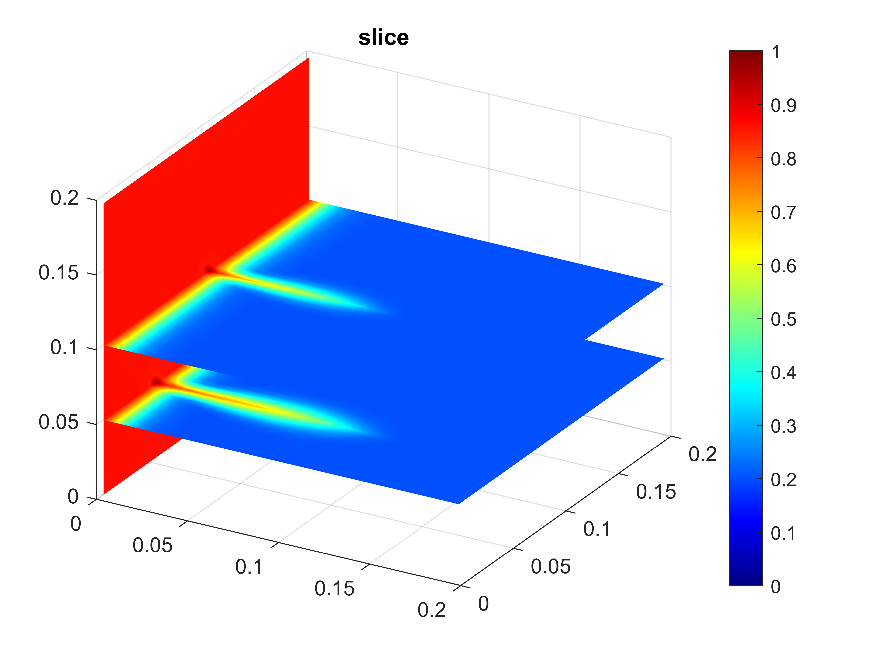}
\caption{ The distributions of porosity for Example 5 with Neumann boundary condition for temperature.}  \label{fig: porosity_example 5}
\end{figure}

The distributions of porosity for Examples 3-5 in 2- and 3-D cases are presented in Figures \ref{fig: porosity_example 3}-\ref{fig: porosity_example 5} respectively. These results are computed on the grid of 80 $\times$ 80 cells in 2-D case and 40 $\times$ 40 cells in 3-D case.  It can be easily presented that the heterogeneity of porosity and permeability in wormhole formations have significant influence on the wormhole formation dynamics, which promotes the non-uniformity of the chemical reaction. Besides the average porosities in all frameworks are increasing, which reveals that the matrix is eaten by the acid. 

\section{Concluding remarks}
In this paper, we developed a fully decoupled and linear scheme for the wormhole model  with heat transmission process on staggered grids,  which only requires solving a sequence of linear elliptic equations at each time. An error analysis for the velocity, pressure, concentration, porosity and temperature in different norms is established rigorously. Finally, we presented numerical experiments in two- and three-dimensional cases to 
verify the theoretical analysis and effectiveness of the constructed scheme.

\bibliographystyle{siamplain}
\bibliography{wormhole}

\end{document}